\documentclass[10pt,leqno,a4paper]{article}
\usepackage{amsmath} 
\usepackage{amsthm}
\usepackage{amssymb}
\usepackage{color}
\usepackage{tikz}
\usepackage{tikz-cd}
\usetikzlibrary{matrix,arrows,decorations.pathmorphing}

\usepackage[english]{babel}
\usepackage[totalheight=22 true cm, totalwidth=14 true cm]{geometry}
\usepackage{setspace}
\usepackage{mathrsfs}
\usepackage[all]{xy}

\newtheorem{thm}{Theorem}[section]
\newtheorem{cor}[thm]{Corollary}
\newtheorem{lemma}[thm]{Lemma}

\newtheorem{prop}[thm]{Proposition}

\newtheorem{conv}[thm]{Convention}

\newtheorem{defn}[thm]{Definition}

\theoremstyle{remark}

\theoremstyle{definition}

\newtheorem{rmk}[thm]{Remark}
\newtheorem{exa}[thm]{Example}

\newtheorem{notation}[thm]{Notation}
\numberwithin{equation}{thm}
\def\beq{\begin{equation}}
\def\eeq{\end{equation}}
\def\beqa{\begin{equation*}}
\def\eeqa{\end{equation*}}
\def\ben{\begin{enumerate}}
\def\een{\end{enumerate}}
\def\besp{\begin{split}}
\def\eesp{\end{split}}

\def\crash#1{}

\def\Z{{\mathbb Z}}
\def\0{{\mathbb O}}

\def\Q{{\mathbb Q}}
\def\R{{\mathbb R}}

\def\F{{\mathbb F}}

\def\1{{\mathbf 1}}

\def\l{\left}
\def\r{\right}
\def\[[{\l[\l[}
\def\]]{\r]\r]}

\def\fpm{{\rm pm}}
\def\CW{{\rm CW}}

\def\ie{\emph{i.e.}\,}

\def\rmW{{\rm W}}

\def\cA{{\mathcal A}}
\def\cB{{\mathcal B}}

\def\cF{{\mathcal F}}
\def\cG{{\mathcal G}}
\def\cI{{\mathcal I}}
\def\cM{{\mathcal M}}

\def\cH{{\mathcal H}}

\def\cL{{\mathcal L}}

\def\cP{{\mathcal P}}
\def\cR{{\mathcal R}}
\def\cS{{\mathcal S}}
\def\cT{{\mathcal T}}

\def\cU{{\mathcal U}}

\def\g"{``}
\def\g'{`}

\def\sC{{\mathscr C}}

\def\sH{{\mathscr H}}
\def\sJ{{\mathscr J}}

\def\sL{{\mathscr L}}

\def\sP{{\mathscr P}}

\def\what{\widehat}

\def\sp{{\rm sp}}

\def\can{{\rm can}}

\def\for{{\rm for}}

\def\CW{{\rm CW}}
\def\BW{{\rm BW}}

\def\Bil{{\rm Bil}}
\def\Ril{{\rm Ril}}

\def\hom{{\rm Hom}}

\def\ind{{\rm ind}}

\def\unif{{\rm unif}}

\def\id{{\rm id}}
\def\Ker{{\rm Ker}}
\def\Coker{{\rm Coker}}
\def\Im{{\rm Im}}
\def\Coim{{\rm Coim}}

\def\ul{\underline}
\def\ol{\overline}
\def\iso{\xrightarrow{\ \sim\ }}
\def\map#1{\xrightarrow{\ #1\ }}

\def\limpro{\mathop{\lim\limits_{\displaystyle\leftarrow}}}

\def\limind{\mathop{\lim\limits_{\displaystyle\rightarrow}}}
\def\limIND#1{\mathop{\lim\limits_{\substack {\displaystyle\rightarrow \\ #1} } }\,}
\def\limPRO#1{\mathop{\lim\limits_{\substack {\displaystyle\leftarrow \\ #1} } }\,}

\def\wt{\what{\otimes}}

\def\bd{{\rm bd}}

\author{Francesco Baldassarri
\thanks{Universit\`{a} di Padova,
Dipartimento di Matematica, Via Trieste, 63, 35121 Padova, Italy.}
 }

\title{Non archimedean  gauge seminorms}

\begin{document} 
\date{\today}

\maketitle 
\begin{abstract}  
We fix a  base commutative topological ring $k$, separated and complete in a linear topology. 
Within  the  category $\cL\cM_k$  of $k$-linearly topologized $k$-modules, we single-out the full subcategory $\cL\cM^u_k$ of $k$-modules whose scalar product is 
uniformly continuous. We describe limits and colimits, and introduce a  tensor product $\wt_k^c$ (resp. $\wt_k^u$) in $\cL\cM_k$ (resp. in $\cL\cM^u_k$).
When $k=K^\circ$, for  a non trivially valued  non archimedean field $K$, $K$-Banach spaces \cite{schneider} are objects of $\cL\cM_k$ but not of $\cL\cM^u_k$. 
We propose a definition of a \emph{pseudobanach} $k$-module which coincides with the one of a $K$-Banach space  if $k = K^\circ$, but covers in general the notion of a family of 
Banach spaces over  variable fields. We describe 
the category $\cL\cR_k$ (resp. $\cR\cR_k$) of complete $k$-linearly (resp. linearly) topologized  $k$-rings and the full subcategory  $\cL\cR_k^u$ of  $\cL\cR_k$ of the $k$-rings for which the 
scalar product is uniformly continuous. 
We discuss limits and colimits in $\cL\cR_k$ (resp. $\cL\cR^u_k$, resp. $\cR\cR_k$) and examine their commutation with the monoidal structures $\wt_k^c$ (resp. $\wt_k^u$).
The former monoidal structure is analog to both Schneider's $- \wt_{K,\iota} -$ and  $- \wt_{K,\pi} -$ \cite{schneider}, while the latter is the one used in the theory of formal schemes. 
 \end{abstract}
 
\tableofcontents
\bigskip
\setcounter{section}{-1}
 \begin{section}{Introduction}
The main motivation of this paper is that of providing solid foundations to the theory of  commutative group and ring functors  on categories of commutative topological $k$-rings, over some fixed 
base ring $k$ 
which is complete in a  linear topology.  In particular, the present paper is preliminary to \cite{psi0}. 
\par \smallskip
Recent developments of arithmetic geometry, in particular, use variants $\rmW$, $\CW$, $\BW$ of the functors of Witt vectors, covectors and bivectors defined for topological rings of characteristic $p>0$    \cite{MA}, \cite{fontaineA}, \cite{FF}, 
to establish the remarkable \emph{tilting equivalence} of Scholze \cite{scholze}. Such Witt-type functors can also  be globalized so that to apply to (non-archimedean) analytic  spaces and to extend 
geometrically the previous equivalence to relative situations \cite{KL1}, \cite{KL2}. 
 The discussion of the most general type of topological rings to which these functors apply is however usually avoided. In general, one   
considers Banach rings rather than  rings  
complete in a linear topology. On the other hand, a glance at the literature indicates that 
little more than definitions are to be found about
the category $\cL\cR_k$ (resp. $\cR\cR_k$) 
of $k$-rings which are complete in a $k$-linear (resp. linear) topology, as soon as one leaves the safe continent of Noetherian adic rings or of mild variations of such. 
\par
From another viewpoint, a rich theory of locally convex topological vector spaces over a non-archimedean field $K$ exists \cite{schneider}, 
and the correspondence 
\par \medskip \noindent
\centerline {open lattice $\longleftrightarrow$ gauge seminorm }
\par \medskip \noindent
establishes a link between the additional information provided by a seminorm and the one obtained by regarding  a topological $K$-vector space as linearly topologized complete $k$-module, 
where $k = K^\circ$,  the ring of integers of $K$. We point out however that such rings $k$ are quite special. In particular they are essentially one-dimensional, while we are interested 
in higher-dimensional base-rings, as well. 
\par
\smallskip
We are lead to single-out within the category $\cL\cM_k$ of $k$-linearly topologized separated and complete topological $k$-modules $M$, for which the scalar product $k \times M \to M$ is continuous for the product topology, the full subcategory  $\cL\cM^u_k$ consisting of \emph{uniform} objects, namely those $M$ for which the scalar product is \emph{uniformly} continuous for the product uniformity of $k \times M$. 
Notice that, unless $K$ is trivially valued, a non-zero $K$-Banach space is never a uniform object of $\cL\cM_k$, for $k = K^\circ$.
On the other hand, a standard assumption 
in the theory of, say, $k$-formal schemes topologically locally of finite type over a Noetherian $k$, is that the $k$-linear topologies on any $k$-module $M$ considered 
should be weaker than the topology induced by $k$ (meaning 
the \emph{naive canonical topology} of $M$, Definition~\ref{naivedef}). This is precisely  the meaning of $M$ being uniform, see Lemma~\ref{factexact0}. 
Such an assumption cannot be made 
in our context since it does not generally hold for a non-archimedean  field $K$ itself, on which $k=K^\circ$ induces the trivial topology, namely $\{\emptyset, K\}$. Moreover we cannot make any finiteness 
assumption, and $k$ in particular may not be Noetherian. 
\par
The existence of the full subcategory  $\cL\cM^u_k$ of $\cL\cM_k$
generates two distinct notions $\wt^u_k$ and $\wt^c_k$ of topological tensor product, where the apex ``$u$'' refers to ``uniform'' while the apex ``$c$'' refers to ``continuous''. 
The monoidal structure $\wt^c_k$ is the completion of the one used for Fr\'echet spaces in \cite{schneider}, and there denoted $\otimes_{K,\pi} = \otimes_{K,\iota}$, 
 while $\wt^u_k$ is the monoidal structure used in the theory of $k$-formal groups topologically of finite type. 
 \par
\smallskip
The purpose of this paper is  twofold. On the one hand, we want to encompass the theory of Banach vector spaces over a non archimedean 
field $K$, and of continuous $K$-linear maps, within the theory of $k$-linearly topologized separated and complete topological $k$-modules, where $k = K^\circ$. 
In particular, we want to characterize $K$-Banach spaces within such topological $K^\circ$-modules, with no 
 reference to a norm. Notice that $k$ is a linearly topologized ring in the sense that 
a fundamental system of neighborhoods of $0$ in $k$ consists of open ideals. We obtain the notion of a \emph{pseudobanach} $k$-module or $k$-algebra (see Definition~\ref{pseudocan} and Definition~\ref{banringdef} below). 
When $k = K^\circ$ for a non trivially valued complete non-archimedean field, the full subcategory $\cP\cB_k$ (resp. $\cP\cB\cA_k$ (resp. $\cU\cP\cB\cA_k$)) of 
$\cL\cM_k$ (resp. $\cL\cR_k$) of pseudobanach $k$-modules 
(resp. of commutative pseudobanach $k$-algebras (resp. of $\fpm$-type, see Definition~\ref{banringdef} below)) is equivalent to the category of $K$-Banach spaces (resp. of commutative $K$-Banach algebras (resp. of $\fpm$-type \cite{fontaineA})) and 
continuous $K$-linear homomorphisms. This part of our discussion may be seen as a non-archimedean analog of 
the theory of gauge seminorms, as Schneider's \cite{schneider}, with the difference that it is developed over $k$, with no reference to $K$. This, by the way, accounts for the title we chose for this paper. 
For general $k$, 
the categories $\cP\cB_k$, $\cP\cB\cA_k$ and $\cU\cP\cB\cA_k$ are new.  
\par
The monoidal structure  $-\wt^c_k C$ is used  to define base change for $K$-Banach algebras (viewed as $k=K^\circ$-modules) via a  morphism $k \to C$ of $\cR\cR_k$.  The first  
main result of this paper is 
  Proposition~\ref{basechange1} which shows that, under the mild condition {\bf OPW} on $k$ and $C$, if $C/k$ is pro-flat (see Definition~\ref{tffdef}), then 
for any $k$-pseudobanach space $M$, $(M)^c_{C} = M\wt^c_k C$ is a $C$-pseudobanach space.
\par
\smallskip
On the other hand, for application to the representability of  our group and ring functors 
we need to have at our disposal a full subcategory $\cI\cR\cR_k$ of $\cL\cR_k$ containing the algebras representing 
all functors of interest to us, together with a base-change functor $(-)^\ind_{C} : \cI\cR\cR_k \to \cI\cR\cR_C$
for  $k \to C$ as before. To give an idea of the difficulties we faced, we point out that our group or ring functors are  
defined on the entire category $\cL\cR_k$. They are represented by Hopf or bi-Hopf algebra objects of $\cL\cR^u_k := \cL\cR_k \cap \cL\cM_k^u$ which    however are  
\emph{inductive limits in $\cL\cR_k$} (hence also in $\cL\cR^u_k$), denoted ${\limIND{\alpha}}^\cL R_\alpha$ (resp. ${\limIND{\alpha}}^u R_\alpha$), 
\emph{of inductive systems $\{R_\alpha\}_\alpha$  in $\cR\cR_k$}. This property defines the full subcategory $\cI\cR\cR_k$ of $\cL\cR_k$. 
The basic linear algebra constructions involved require special attention and raise the general problem of existence and description of limits and colimits, and 
of suitable monoidal structures.  In particular, assume one of our functors $F$ is represented by an object $\sL_k = {\limIND{\alpha}}^\cL R_\alpha$ in $\cL\cR^u_k$, as before, in the sense that
$$
F(X) = \hom_{\cL\cR_k}(\sL_k,X) \;,
$$
for any $X$ in $\cL\cR_k$.  Then  the second main result of this 
paper, Proposition~\ref{scalrestr}, shows that 
the restriction   of $F$ to $\cL\cR^u_C$, where $C$ is an object of $\cR\cR_k$, is represented by an object $\cL_C := (\cL_k)^\ind_{C}$ of $\cL\cR^u_C$ which is also an  
inductive limit, namely  ${\limIND{\alpha}}^\cL (R_\alpha)^u_{C}$, in $\cL\cR_C$ of an 
inductive system in $\cR\cR_C$. Here, for $X$ in $\cR\cR_k$,  $(X)^u_{C} = X \wt^u_k C$ is an object of $\cR\cR_C$, see \eqref{prodalg1} and   \eqref{tensudef}. 
This at least shows that the functor 
\beqa
\begin{split}
( -)^\ind_{C} : \cI\cR\cR_k & \longmapsto \cI\cR\cR_C \\
\sL_k & \longmapsto \sL_C = (\sL_k)^\ind_{C}
\end{split}
\eeqa
is well-defined (see Remark~\ref{scalrestr1} below).
\par \smallskip
It will be shown  in \cite{psi0} that the colimit $\cT(\sL_k) = {\limIND{\alpha}}^\cR R_\alpha$ (see Lemma~\ref{indlim1}), \emph{taken this time in $\cR\cR_k$},  of the same 
inductive system $\{R_\alpha\}_\alpha$  in $\cR\cR_k$ which defines $\sL_k$, represents, under suitable conditions,  the subfunctor $X \mapsto F^\bd(X)$ of bounded elements of $F(X)$. 
The importance of finding a topological algebra representing the subfunctors $\rmW^\bd$,  $\CW^\bd$ and $\BW^\bd$ of \cite{fontaineA}, \cite{FF} can hardly be overestimated. 
\par
\medskip

For applications it is important to consider the restriction to $\cU\cP\cB\cA_k$ of functors $F$ of the type 
described above.  We recall that, when $k = K^\circ$, as above, $\cU\cP\cB\cA_k$ identifies with the category of $K$-Banach algebras of $\fpm$-type of \cite{fontaineA}. 
In \cite{psi0} we will define the functors $\CW$, $\BW$, $\rmW^\bd$,  $\CW^\bd$ and $\BW^\bd$ on $\cL\cR_k$ 
when $k$ is of characteristic $p$ and is \emph{topologically perfect}, in the sense that its Frobenius is an automorphism. We will prove that $\CW^\bd$ (resp. $\BW^\bd$) induces a functor $\cU\cP\cB\cA_k \to \cP\cB_{\rmW (k)}$ 
(resp. $\cU\cP\cB\cA_k \to \cU\cP\cB\cA_{\rmW (k)}$). 
Of special interest is the case of $k=\kappa = \F_p[[t^{p^{-\infty}}]]$, completed in  the $t$-adic topology. 
In that case, for $R$  an object of 
$\cU\cP\cB\cA_\kappa$, $\CW^\bd_\kappa(R)$ (resp. $\BW^\bd_\kappa(R)$, for $R$ topologically perfect) is an object of $\cP\cB_{\rmW (\kappa)}$ (resp. $\cU\cP\cB\cA_{\rmW (\kappa)}$), where $\rmW (\kappa)$ is of dimension 2. 
Then, for any quotient map $\rmW (\kappa) \to L^\circ$ onto the ring of integers of a perfectoid $p$-adic field $L$, $(\CW^\bd_\kappa(R))^c_{L^\circ}$ (resp. $(\BW^\bd_\kappa(R))^c_{L^\circ}$) (see Proposition~\ref{basechange1} (resp. Corollary~\ref{basechunif}) below)
becomes  an $L$-Banach space (resp. a perfectoid algebra over $L$ \cite{scholze}) in the usual sense. 

  \par \smallskip
  The generalities we develop in this paper  suffice for  the  purpose we have in mind, and in particular provide 
a natural  framework to \cite{psi0}. 
It appears however  that the discussion of linear topologies, and of  ``linear uniformities'', 
  should be developed more systematically,  especially as far as limits and colimits are concerned. It seems that such a 
 systematic discussion does not exist in the literature yet. 
  \par \medskip
   \emph{Acknowledgments.} I am very grateful to Maurizio Cailotto for generously dedicating a lot of his time to discuss with me a variety of mathematical 
 questions which arose during the preparation of this article.   I had the privilege of discussing the content of this paper with Peter Schneider; his criticism was very useful.  
\end{section}
\begin{section}{Notation}
A prime number $p$ is fixed throughout this paper even though it only appears  in examples; then $\Z_p$ and $\Q_p$ have the usual meaning. 
Unless otherwise specified, a  \emph{ring}  is meant to be commutative with 1. 
We denote by $\cA b$ (resp. $\cR ings$, resp. $\cM od_k$) the category of abelian groups (resp. of commutative rings with 1, resp. of unitary $k$-modules for $k$ a ring). 
Generally speaking, for any ring $R$ in $\cR ings$, the $R$-algebras appearing in this paper will be understood to be  commutative with 1. 
All the non-archimedean fields $K = (K,v_K)$ we consider will be complete but not necessarily non trivially valued.  
We set $K^\circ := \{x \in K\,|\,v_K(x) \geq 0\,\}$ and $K^{\circ \circ}:= \{x \in K\,|\,v_K(x)  > 0\,\}$. 
\end{section} 
\begin{section}{Linearly topologized modules}
A topological ring $k$ is \emph{linearly topologized} if it has a basis  of open neighborhoods of 0 consisting of ideals. This implies that the product map
$$
\mu_k:  k \times k \longrightarrow k \;\;,\;\; (x,y) \longmapsto xy
$$
is uniformly continuous.  
All over this paper,  $k$ will be a complete separated linearly topologized topological ring.  
We often call $\cP(k)$ a fundamental system of open ideals of $k$. 
For certain constructions, we will need the further assumption  of \emph{openness} on a linearly topologized ring $R$, namely
\par \medskip
{\bf (OP)} \emph{For any regular element $a$ of $R$, the
map 
$$ R \longrightarrow R \;\;,\;\; x \longmapsto a \, x
$$
is open.}
\par
\medskip
\begin{rmk} \label{nonOP} A typical example of a linearly topologized ring $R$ for which property {\bf OP} fails is $\Z[T]$ equipped with the $(p,T)$-adic topology. It is clear in fact that the ideals $(T^N)$ and 
$(p^N)$ are not open in $\Z[T]$, for any $N \in \Z_{\geq 1}$. The condition is instead verified by any complete rank one valuation ring.
\end{rmk}
\begin{rmk} \label{OP} It follows from assumption {\bf OP} that, if $R$ is a domain,  
 the product map
$$
\mu_R : R \times R  \longrightarrow R  \;\;,\;\; (x,y) \longmapsto x \, y
$$
is open. In fact, for any open ideals $I,J$ of $R$, $IJ = \bigcup_{a \in I -\{0\}} aJ$ is a union of open subsets of $R$, hence is open. Now, for any open subsets $U ,V \subset R$, $U = \bigcup_{a \in U} a+I_a$, 
$V = \bigcup_{a \in v} a+J_a$ where $I_a,J_a$ are open ideals of $R$. Then 
$$UV = \bigcup_{a \in U, b \in V} (a +I_a)(b+J_b) =  \bigcup_{a \in U, b \in V} ab +bI_a + aJ_b +I_a J_b
$$
is a union of open subsets of $R$. 
\end{rmk}
A weaker assumption  of  openness on a linearly topologized ring $R$ is
\par \medskip
{\bf (OPW)} \emph{For any open ideals $I$ and $J$ of $R$, the ideal $IJ$ is open.}
\par
\medskip
Assumption {\bf OPW} holds for any adic ring \cite[Chap. 0, Def. 7.1.9]{EGA}. 
Whenever assumption {\bf OP} or  {\bf OPW} will be needed,  the reader will be explicitly warned.  Assumptions {\bf OP} and {\bf OPW} certainly hold for $R$ if 
$R$ is   
endowed with the discrete topology. Let $R$ be a linearly topologized ring. A \emph{topological $R$-module} is a topological abelian group which is also an $R$-module  such that the scalar product map
$$
R \times M \longrightarrow M  \;\;,\;\; (a,x) \longmapsto a \, x
$$
is continuous \emph{for the product topology of $R \times M$}. 
An  \emph{$R$-linear}  topology on an  $R$-module $M$  is an  
$R$-module topology  
which has a   basis of open neighborhoods of 0 consisting of $R$-submodules of 
$M$. We often call $\cP(M)$ a fundamental system of open $R$-modules in $M$. Then $M$ is equipped with a canonical uniform structure  and we say that it is \emph{uniform} 
if the scalar product map is uniformly continuous for the product uniformity of $R \times M$.  Notice that if the topology of $k$ is the discrete one, then any 
 topological $R$-module is uniform.  
A  complete $k$-linearly topologized $k$-module $M$  is meant to be separated. Similarly, when we refer to ``completion'' we always mean ``separated completion''. 
A  relevant condition for an $R$-linearly topologized $R$-module $M$ is 
\par \medskip
{\bf (OPM)} \emph{For any regular element $a$ of $R$, the
map 
$$ M \longrightarrow M \;\;,\;\; x \longmapsto a \, x
$$
is open.}
\par
\medskip

When discussing topological vector spaces  $V$ over a  non archimedean field $K$, we will generally choose $k = K^\circ$, equipped with the subspace topology and will assume that 
$V$ is equipped with a $k$-linear topology. In this case $k$  satisfies {\bf OP} and $V$  satisfies {\bf OPM} for $R =k$. Let $K$  be non-trivially valued let $k = K^\circ$. Let $V$ be 
a locally convex $K$-vector space in the sense of 
\cite{schneider}, then the  $k$-linearly topologized $k$-module $V$ is a topological $k$-module as well \cite[Lemma 4.1]{schneider}, but it is not uniform, in general.
Mostly we think 
of $k = \F_p$, the field with $p$ elements,  or $= \Z_{(p)}$, either one equipped with the discrete topology, or of $k= \Z_p$ 
equipped with the $p$-adic topology.  
\begin{defn}
\label{LMk}
We let $\cL\cM_k$ (resp. $\cL\cM^u_k$) be the category of  complete $k$-linearly topologized (resp. uniform) topological $k$-modules   and continuous $k$-linear maps.
\end{defn}

We denote by $k^\for$ the ring underlying the topological ring $k$. 
In general $M \mapsto M^\for$ will be the natural forgetful functor $\cL\cM_k \longrightarrow \cM od_{k^\for}$. 
To avoid excessively burdening the notation however, the category $\cM od_{k^\for}$ will be simply denoted by $\cM od_k$. Similarly, we generally 
write $\hom_k$ for $\hom_{k^\for}$, $\Bil_k$ (standing for ``$k$-bilinear'') for $\Bil_{k^\for}$, and shorten $M^\for \otimes_{k^\for} N^\for$ into $M \otimes_k N$. (Complete tensor product will have a distinguished 
notation, anyhow.) 
\begin{defn} \label{naivedef}
The \emph{naive canonical topology} on a $k$-module $M$ is the $k$-linear topology 
 with a basis of open $k$-submodules consisting of $\{I M\}_I$, for $I$ running over the set of open ideals of $k$. 
 For any $M$ in $\cM od_k$, we define the object 
 $M^\can$ of $\cL\cM_k$ to be the completion of $M$ in its naive canonical topology, \ie the $k$-module
\beq
\label{naive00}
\what{M} = \limpro_{I \in \cP(k)} M/IM
  \;,
\eeq 
  equipped with the weak topology of the projections to the discrete $k/I$-modules $M/IM$.  
  \end{defn}
 \begin{lemma} \label{factexact0} Let $M$ be an object of  $\cL\cM_k$. Then   $M$ is uniform if and only if its topology is weaker than the naive canonical topology.
\end{lemma}
  \begin{proof} Assume $M$ is 
  uniform. Then, for any $U \in \cP(M)$ there is a $V \in \cP(M)$ and an $I \in \cP(k)$ such that 
  for any $a \in k$ and $m \in M$,  
  $$(a+I)(m + V) \subset am + U \;.
  $$
  But this implies that $IM \subset U$, so that the topology of $M$ is weaker than the canonical one. The converse 
is clear.   
  \end{proof}
   \begin{cor} \label{factexact1} For any $M$ in $\cM od_k$,  $M^\can$ is an object of $\cL\cM^u_k$. 
 \end{cor} 
     \begin{defn} \label{canondef} The objects of $\cL\cM^u_k$ of the form $M^\can$, for an $M$ in $\cM od_k$, are said to be  \emph{canonical} 
  (resp. $k$-\emph{canonical}, for more precision) or to have the (resp. $k$-)\emph{canonical topology}.
  \end{defn}
 \begin{rmk} \label{naive0} Notice that the  naive canonical topology on a $k$-module $M$  runs, in general, into a serious difficulty. Namely, it 
 is not true in general that the completion of $(M, \{I M\}_I)$, that is the object $\what{M}$ defined above, would still carry the naive canonical topology. We are indebted to Peter Schneider 
 for pointing out this problem. 
 \end{rmk}
 \begin{prop} \label{adjoint0} The functors 
\beq
\begin{tikzcd}
\cM od_k   \arrow[bend left]{rr}{\can}&&\cL\cM^u_k \arrow{ll}{\rm for}
  \end{tikzcd}
  \eeq
  are adjoint~: for any  $M$ in $\cM od_k$ and $N$ in $\cL\cM^u_k$, there are canonical identifications
\beq \label{adjoint}
  \hom_{\cL\cM_k}(M^\can,N) = \hom_{\cM od_k}(M,N^{\rm for}) \; .
\eeq

\end{prop}
\begin{proof}
 For any $N$ in $\cL\cM^u_k$  we have a canonical continuous $k$-linear map
\beq \label{canmor}
(N^\for )^\can  \longrightarrow N \; .
\eeq  
So, for any $M$ in $\cM od_k$ and $k$-module morphism $f:M \to N^\for$, we get $\what{f} : \what{M} \to \what{N^\for}$, hence a morphism $f^\can: M^\can \to (N^{\rm for})^\can \to N$.
  Conversely, from $g: M^\can \to N$ we obtain $M \to ( M^\can)^\for \map{g^\for} N^\for$. 
  \end{proof}

\begin{lemma} \label{limcolim0} \hfill
\ben
 \item The category  $\cL\cM_{k}$ admits  limits. Its full subcategory $\cL\cM^u_{k}$  is closed by limits. 
\item  The category $\cL\cM_{k}$ admits  colimits.   
\item  The category $\cL\cM^u_{k}$ admits  colimits. 
\item  For a finite inductive system in $\cL\cM^u_{k}$, the two colimits in $\cL\cM^u_{k}$ and in $\cL\cM_{k}$ coincide. 
\een
\end{lemma}
\begin{proof}

$\mathit 1.$  
Let  $(M_\alpha)_{\alpha \in A}$ be a projective system  in $\cL\cM_k$ indexed by the 
preordered set $A$. Its projective  limit  in $\cL\cM_k$ is simply 
the projective limit $M' = \limPRO{\alpha \in A} M_\alpha^\for$ equipped with the weak topology $\tau$ of the canonical projections 
$\pi_\alpha : M' \to M_\alpha$. In fact, let us show first that  $(M',\tau)$ is an object of $\cL\cM_k$.  We pick 
$a \in k$, an element $m = (m_\alpha)_\alpha \in M'$ and $U = \pi_{\alpha_0}^{-1}(U_{\alpha_0}) \in \cP(M')$, for some $\alpha_0 \in A$. 
By continuity of the scalar product of $M_{\alpha_0}$, there exist $I_{a,m,\alpha_0} \in \cP(k)$ and  $V_{a,m,\alpha_0} \in \cP(M_{\alpha_0})$
such that  
\beqa \begin{split} 
 (a+ I_{a,m,\alpha_0}) &(m +  \pi_{\alpha_0}^{-1}(V_{a,m,\alpha_0}  )) =  
\pi_{\alpha_0}^{-1}((a+ I_{a,m,\alpha_0}) (m_{\alpha_0} + V_{a,m,\alpha_0} )) \\ & \subset 
   \pi_{\alpha_0}^{-1}(am_{\alpha_0} + U_{\alpha_0})  \subset am  + U \;.
\end{split}
\eeqa
We now observe that if $(M_\alpha)_{\alpha \in A}$ is a projective system  in $\cL\cM^u_k$, for the given 
$U = \pi_{\alpha_0}^{-1}(U_{\alpha_0}) \in \cP(M')$, we may pick 
$I_{a,m,\alpha_0} = I_{\alpha_0} \in \cP(k)$ and  $V_{a,m,\alpha_0} = V_{\alpha_0} \in \cP(M_{\alpha_0})$ independent 
of $a,m$. Then for $I_U :=  I_{\alpha_0}$ and $V_U := \pi^{-1}_{\alpha_0}(V_{\alpha_0})$, we have 
$$
(a + I_U)(m +V_U) \subset am +U 
$$
for any $a \in k$ and $m \in M$, so that the scalar product of $M'$ is uniformly continuous. It is clear   that $(M',\tau)$ is indeed the projective limit of $(M_\alpha)_{\alpha \in A}$ in $\cL\cM_k$.
\par $\mathit 2.$ Let $(M_\alpha)_{\alpha \in A}$ be an inductive system  in $\cL\cM_k$ indexed by the 
preordered set $A$. Its inductive  limit  in $\cL\cM_k$ is calculated as follows. We first consider $M' = \limIND{\alpha \in A} M_\alpha^\for$ in $\cM od_k$ and let $j_\alpha: M^\for_\alpha \to M'$ 
be the natural morphisms.   
We then give to $M'$ the finest $k$-linear topology  such that all maps $j_\alpha: M_\alpha \to M'$ are continuous. So, a basis of open $k$-submodules in $M'$ consists of 
the $k$-submodules $U$ of $M'$  such that
$U_\alpha := j_\alpha^{-1}(U)$ is open in $M_\alpha$, for any $\alpha \in A$. 
Then $\limIND{\alpha \in A} M_\alpha$ is represented by the completion $M$ of $M'$ in that 
topology, equipped with the natural morphisms $i_\alpha : M_\alpha \to  M$ deduced from the $j_\alpha$'s. 
It is clear that, for any fixed $a \in k$,  a scalar product map 
$$
 M \longrightarrow M \;\;,\;\;  x \longmapsto ax
$$ 
is uniquely defined as the inductive limit of the maps 
$$
 M_\alpha \longrightarrow M_\alpha \;\;,\;\;  x \longmapsto ax  
$$
and is then continuous.   
We must check that the scalar product of $M$ is continuous for the product topology. So, let $a \in k$, $m \in M$, and let $U \in \cP(M)$ be as before;  
let $\{m_\alpha\}_{\alpha \in A}$, with 
$m_\alpha \in j_\alpha(M_\alpha)$,  be an $A$-net converging to $m \in M$. 
 So, there exists an index $\alpha_0 \in A$ such that $m_\alpha \in m + U$, for any $\alpha \geq \alpha_0$. We then pick $J \in \cP(k)$ such that 
 $J m_{\alpha_0} \in j_{\alpha_0}(U_{\alpha_0})$ ($\subset U$). 
 Then
 $$
 (a +J)(m+U) \subset am + aU +Jm + JU \subset am + J m_{\alpha_0} + JU \subset am +U \;.
 $$
 We have to prove that $M$ is indeed the inductive limit of the system $\{M_\alpha\}_\alpha$.  
For any $N$ in $\cL\cM_k$,
\beq \label{Homindlim} \begin{split}
&\hom_{\cL\cM_k}(M, N) 
= \{ \varphi \in \hom_{\cM od_k}(\limIND{\alpha \in A} M^\for_\alpha, N) \,|\, \varphi \circ j_\alpha \;\mbox{is continuous}\; \forall \,\alpha \in A\,\} =
\\
& \{ \varphi =  ( \varphi _\alpha)_{\alpha \in A} \in \limPRO{\alpha \in A} \hom_{\cM od_k}(M_\alpha^\for,N) \,|\, \varphi_\alpha \;\mbox{is continuous}\; \forall \,\alpha \in A\,\} =
\limPRO{\alpha \in A} \hom_{\cL\cM_k}(M_\alpha,N) \; .
\end{split}
\eeq  
 \par $\mathit 3.$ Suppose $(M_\alpha)_{\alpha \in A}$ is an inductive system  in $\cL\cM^u_k$ indexed by the 
preordered set $A$. Then we slightly modify the discussion of $\mathit 2$  in that we equip $M'$ the finest $k$-linear topology, \emph{weaker than the naive 
canonical topology},  such that all maps $j_\alpha: M_\alpha \to M'$ are continuous.
Then the proof of $\mathit 2$ can be adapted to the present situation. 
 \par $\mathit 4.$ Is clear by the construction. 
\end{proof}

\begin{notation} \label{uniflim} For an inductive system  $(M_\alpha)_{\alpha \in A}$ in $\cL\cM_k$ (resp. in $\cL\cM^u_k$), the inductive limit 
of $(M_\alpha)_{\alpha \in A}$ in $\cL\cM_k$ (resp. in $\cL\cM^u_k$) 
will be denoted  $\limIND\alpha M_\alpha$ (resp. 
${\limIND\alpha}^u M_\alpha$). If $(M_\alpha)_{\alpha \in A}$  is an inductive system in $\cL\cM^u_k$, we have a canonical surjective morphism 
\beq  \label{uniflim0} 
\limIND\alpha M_\alpha \longrightarrow {\limIND\alpha}^u M_\alpha
\eeq 
in  $\cL\cM_k$.
\end{notation}
\begin{rmk} \label{uniflimdiscr}
If $k$ is discrete an inductive system $(M_\alpha)_{\alpha \in A}$  in $\cL\cM_k$  is also an inductive system in $\cL\cM^u_k$ and
\eqref{uniflim0}  is an isomorphism.
\end{rmk}
\begin{notation} \label{cohnhd}
Let $(M_\alpha)_{\alpha \in A}$ be 
an inductive system  in $\cL\cM_k$ with transition morphisms $j_{\alpha,\beta}:M_\alpha \to M_\beta$ for $\alpha \leq \beta$. 
For any $\alpha \in A$ let $\cP(M_\alpha)$ denote the set of open $k$-submodules of $M_\alpha$. Then 
\emph{a coherent system  of open $k$-submodules of $(M_\alpha)_{\alpha \in A}$} is a system $\sP := (P_\alpha)_{\alpha \in A}$ such that for any $\alpha \leq \beta$ in $A$, 
$j_{\alpha,\beta}^{-1}(P_\beta) = P_\alpha$. The set $\sC((M_\alpha)_{\alpha \in A})$ of coherent systems of open $k$-submodules of $(M_\alpha)_{\alpha \in A}$ forms a filter
of $k$-submodules of $\prod_{\alpha \in A} M_\alpha^\for$. 
\end{notation}
\begin{lemma} \label{expldirlim} We use the notation of \eqref{cohnhd}.
\hfill \ben
\item Let $(M_\alpha)_{\alpha \in A})$ be 
an inductive system  in $\cL\cM_k$. Then
\beq  \label{expldirlim1}
\limIND\alpha M_\alpha = \limPRO{\sP \in \sC((M_\alpha)_{\alpha \in A})}  M_\alpha/P_\alpha \;.
\eeq
\item Let 
$(M_\alpha)_{\alpha \in A}$ is an inductive system  in $\cL\cM^u_k$. Then 
\beq  \label{expldirlim2}
{\limIND\alpha}^u M_\alpha = \limPRO{I \in \cP(k)} \limIND{\alpha \in A} M_\alpha/ \ol{I M_\alpha} \;, 
\eeq
in which
$$
\limIND{\alpha \in A} M_\alpha/ \ol{I M_\alpha}  = {\limIND{\alpha \in A}}^u M_\alpha/ \ol{I M_\alpha} 
$$
since $k/I$ is discrete. 
\een
\end{lemma}
\begin{proof} Clear.
\end{proof}
\begin{exa} \label{difflim} Consider the inductive system 
$$
(\Z_p,p) := \Z_p \map{p \cdot} \Z_p \map{p \cdot} \Z_p \map{p \cdot} \dots \; ,
$$
where $p \cdot : \Z_p \to \Z_p$ is multiplication by $p$. So, $(\Z_p,p)$ is an inductive system in $\cL\cM_k$, for $k = \Z_p$. Then
$$
\limind (\Z_p,p) = \Q_p \;\;\mbox{while} \;\; {\limind}^u (\Z_p,p) = (0) \;.
$$
\end{exa}
\begin{rmk} \label{openind} 
 We conclude from  Proposition~\ref{adjoint0} that  the functor 
 \beqa \begin{split}
 (-)^\can : \cM od_k & \longrightarrow  \cL\cM_k^u \\
 M &\longmapsto M^\can
\end{split}
 \eeqa
 commutes with inductive limits, that is, for any inductive system  $(M_\alpha)_{\alpha \in A}$ in $\cM od_k$,
 \beq \label{indcomm}
 ({\limIND\alpha} M_\alpha)^\can = {\limIND\alpha}^u M_\alpha^\can \;,
 \eeq
  while $M \mapsto M^{\rm for}$ commutes with the projective ones. Actually, for any projective system  $(M_\alpha)_{\alpha \in A}$ in $\cL\cM_k$,
 \beq \label{projcomm}(\limPRO{\alpha \in A} M_\alpha)^\for  =  \limPRO{\alpha \in A} M_\alpha^\for \;.
\eeq
On the other hand it is clear that,  if for any $\alpha \leq \beta$ the morphism $M_\alpha \to M_\beta$ of an inductive system 
$(M_\alpha)_{\alpha \in A}$   in $\cL\cM_k$ is open, then 
 \beq \label{indcomm1}(\limIND{\alpha \in A} M_\alpha)^\for  =  \limIND{\alpha \in A} M_\alpha^\for 
\eeq
equipped with the topology for which a fundamental system of open $k$-submodules is given by the set of 
$j_\alpha(P)$, for $\alpha \in A$ and for $P$ an open $k$-submodule of  
$M_\alpha$.  
\end{rmk}
\begin{lemma} \label{kercoker} For morphism $f: M \to N$ in $\cL\cM_k$   we have the following description. \hfill
\ben
\item $\Ker (f) = \Ker (f^\for)$, endowed with the subspace topology in $M$.
\item 
$$\Coker (f)
= \limPRO{Q \in \cP(N)} N^\for/(Q^\for + \Im (f^\for)) \;,
$$ 
where every $N^\for/(Q^\for + \Im (f^\for))$ is equipped with the discrete topology.
\item If $f$ is open, then $\Coker (f) = \Coker (f^\for)$ equipped with the discrete topology. 
\een
\end{lemma}
\begin{rmk} \label{imcoim}
The categories $\cL\cM_k$  and $\cL\cM_k^u$ are not in general abelian. In particular, for a morphism $f:M \to N$ in any of the previous categories, we have a canonical mono/epi-morphism
  $\ol{f} : \Coim (f) \to \Im (f)$ which permits to regard $\Coim (f)$ as a dense $k$-submodule of $\Im(f)$, whose topology is finer than the subspace topology induced by 
  the topology of $\Im(f)$, \ie by the one of $N$. When $\ol{f}$ is an isomorphism  
  we say that $f$ is \emph{strict}. 
  \end{rmk}

\begin{defn} \label{sub-obj} 
 A \emph{sub-object} $N$ of some object $M$ in  $\cL\cM^u_k$ or in $\cL\cM_k$  is a sub-object $N^\for$ of $M^\for$ which is closed in $M$ and is equipped with the subspace topology. 
 The morphism $i_N : N \to M$ will be called the  \emph{embedding} 
 of the sub-object $N$ of $M$.   If $i_N$ is a continuous open injection, then $N$ is an \emph{open sub-object} of $M$ and $i_N$ will be called the \emph{open embedding} 
 of $N$ in $M$.   
 \end{defn}
 \begin{exa} \label{bicampoincompleto} Let $\sC (\Q_p,k)$ (resp. $\sC_\unif (\Q_p,k)$ be the $k$-module of continuous (resp. uniformly continuous) functions $\Q_p \to k$, equipped with 
 the topology of simple (resp. uniform) convergence. Then the natural morphism  $\sC_\unif (\Q_p,k) \to \sC (\Q_p,k)$  in $\cL\cM_k$ is injective and has dense image, so is not the 
 embedding of a sub-object of $\sC (\Q_p,k)$ in our sense. 
 \end{exa}
\begin{defn} \label{quotient}   
  If $i_N : N \to M$ is a sub-object of $M$ in $\cL\cM_k$ (resp. $\cL\cM^u_k$) we set
$$M/N = \Coker (i_N)   
$$
in the category $\cL\cM_k$ (resp. $\cL\cM^u_k$). So, 
$$M/N = \limPRO{Q \in \cP(M)} M^\for/(Q^\for + N^\for) \;,
$$ 
where every $M^\for/(Q^\for + N^\for)$ is equipped with the discrete topology.  An \emph{exact sequence} is any  sequence isomorphic to a sequence 
of the form
$$
0 \longrightarrow N \map{i_N} M  \longrightarrow  M/N  \longrightarrow 0 \;. 
$$
 \end{defn}
 \begin{rmk}
The embedding $i_N: N \hookrightarrow M$ of a sub-object is the same as a strict injection, since the latter is necessarily closed. Then 
  $N$ is an open sub-object of $M$ if and only if the morphism $i_N$ is an open injection.
  \end{rmk}
\begin{lemma} \label{openstrict}
Any open map is strict. 
\end{lemma}
\begin{proof}
If $f:M \to N$ is open, then $\Im(f) = \Im(f^\for)$ is an open sub-object of $N$, equipped with the subspace topology. Moreover, any open submodule of $\Coim(f)$ is of 
  the form 
  $\ol{P} := P/\Ker (f)$ where $P \in \cP(M)$ contains $\Ker (f)$. Then $\ol{f}(\ol{P}) = f(P)$, an open subobject of $\Im(f)$. So, $\ol{f}$ is an isomorphism.
\end{proof}

  \begin{rmk} \label{prodiscrete0} Let $M$ be an object of $\cL\cM_k$ whose topology is discrete. Then,
 \beq \label{prodiscrete00} 
 M = \bigcup_{I \in \cP(k)} M_{[I]} 
 \eeq
where  
$$ M_{[I]} := \{m \in M \,|\, am = 0\,,\,\forall a \in I\,\}\;.
$$
Notice that $M_{[I]}$, but not $M$ in general,  is an object of $\cL\cM^u_k$. More precisely we have:
\par
\emph{An object of $\cL\cM_k$ which carries the  discrete topology is uniform if and only if it is an object of  $\cL\cM_{k/I}$, for some open ideal $I$ of $k$. }
\par  
Notice that by Remark~\ref{openind}  formula \eqref{prodiscrete00} can also be written as
 \beq \label{prodiscrete000} 
 M = \limIND{I \in \cP(k)} M_{[I]} 
 \eeq
the colimit  in $\cL\cM_k$ of a filtered inductive system of discrete objects of $\cL\cM^u_k$ and injections. 
 \end{rmk}
\begin{lemma} \label{limcolim} \hfill \ben 
\item Any object $M$ in $\cL\cM_{k}$ is a projective limit of a filtered projective system of discrete $k$-modules and surjections
 \beq \label{prodiscrete11}
M  = \limPRO{P \in \cP(M)} M/P
 \;.
 \eeq 
More precisely, 
 \beq \label{prodiscrete21}
M = \limPRO{P \in \cP(M)}\limIND{I \in \cP(k)} (M/P)_{[I]}
 \;,
 \eeq
 where any $(M/P)_{[I]}$ is a  discrete $k/I$-module.
 \item Any object $M$ in $\cL\cM^u_{k}$ is a projective limit of a filtered projective system of discrete uniform $k$-modules and surjections 
 \beq \label{prodiscrete2}
M = \limpro_P M/P
 \;, 
 \eeq
 where $P$ runs over a fundamental system of open $k$-submodules. Equivalently, any $M/P$ in \eqref{prodiscrete2} is a  discrete $k/I$-module, for some $I = I_P \in \cP(k)$. 
\item  Let $M, N$ be objects of 
$\cL\cM^u_k$ and let $\cP(M)$, $\cP(N)$ be  fundamental systems of open $k$-submodules in $M$ and $N$, respectively. Then 
\beq \label{modhom}
\hom_{\cL\cM^u_k}(M,N) = \limpro_{Q \in \cP(N)}  \limind_{P \in \cP(M)} \hom_k(M/P,N/Q) 
\eeq
as a $k^\for$-module.  Notice that both $M/P$ and $N/Q$ are $k/I$-modules for some $I \in \cP(k)$, so that    $\hom_k(M/P,N/Q) =  \hom_{k/I}(M/P,N/Q)$ 
for any such $I$. 
\item  Let $M, N$ be objects of 
$\cL\cM_k$ and let $\cP(M)$, $\cP(N)$ be  fundamental systems of open $k$-submodules in $M$ and $N$, respectively. Then 
\beq \label{modhom2}
\hom_{\cL\cM_k}(M,N) = \limPRO{Q \in \cP(N)}  \limind_{P \in \cP(M)} \limPRO{I \in \cP(k)} \hom_{k/I}((M/P)_{[I]},(N/Q)_{[I]}) 
\eeq
as a $k^\for$-module.  
\een
\end{lemma}
\begin{proof} \hfill
\par The statement in $\mathit 1$ means that
for an object $M$ in $\cL\cM_{k}$, the underlying $k^\for$-module $M^\for$  is a projective limit in $\cM od_k$ of a filtered projective system 
 of $k^\for$-modules 
  \beq \label{prodiscrete1}
M^\for = \limPRO{P \in \cP(M)} M^\for/P^\for
 \;.
 \eeq 
  The topology of $M$ is the weak topology of the projections $\pi_{P}:M^\for \to M^\for/P^\for$, where the target is equipped with the discrete topology. All this is clear. 
  The remaining part of $\mathit 1$ follows from Remark~\ref{prodiscrete0}.
\par $\mathit 2$ follows from the definition and from Remark~\ref{prodiscrete0}. 
\par $\mathit 3$, $\mathit 4$ follow from $\mathit 2$ and $\mathit 1$, respectively. 
\end{proof}

\begin{lemma} \label{exactsub} Let $f:N \to M$ be a  morphism in $\cL\cM_k$  and let \eqref{prodiscrete11} be 
the representation of $M$ as a limit of discrete objects where $\pi_P : M \to M/P$ and $\pi_{P,Q} : M/P \to M/Q$ are the canonical projections, for any $P \subset Q$ in $\cP(M)$. 
For any $P \in \cP(M)$ let $\ol{N}_P$  be the set-theoretic image of the natural morphism $f_{/P} : N \to M/P =: \ol{M}_P$. 
Then 
\ben
\item $f$ is the embedding of a closed sub-object if and only if  the morphism 
\beq \label{subpro}
F: N \longrightarrow  \limPRO{P \in \cP(M)} \ol{N}_P
\eeq 
deduced from $\{f_{/P}\}_P$ by the universal property of the projective limit is an isomorphism.
\item Let $N$ be a sub-object of $M$ as in \eqref{subpro}. Then the quotient $M/N$ is
\beq \label{subquot}
M/N \iso \limPRO{P \in \cP(M)} \ol{M}_P/\ol{N}_P \;,
\eeq 
\item Assume $f$ is a closed embedding. Then $f$ is open if and only if there exists $P_0 \in \cP(M)$ such that,  for any $P \subset Q \subset P_0$ in $\cP(M)$, the map $\pi_{P,Q} : M/P \to M/Q$ 
induces a surjection $\ol{N}_P \to \ol{N}_Q$. 
\een
 \end{lemma} 
 \begin{proof}  All the assertions are clear. 
 \end{proof}
 \begin{prop}
\label{canon} An object $M$ of $\cL\cM_k$  is canonical   
if and only if it is a projective limit  of a projective system $\{\ol{M}_I \}_{I \in \cP(k)}$, indexed by $I \in \cP(k)$, where each  $\ol{M}_I$ is a discrete $k/I$-module and, for any $I \subset J$ in $\cP(k)$, the morphism 
$\ol{M}_I \to \ol{M}_J$ is a surjection. So, 
\beq \label{canform}
M = \limPRO{I \in \cP(k)} \ol{M}_I \;.
\eeq 
\end{prop}
\begin{proof} Clear. 
\end{proof}
\begin{rmk} \label{opencanon} We observe that an open sub-object $N$ of a canonical module $M$ carries the canonical topology. In fact,
$N$ correspond to a projective system $\ol{N}_I$ of $k/I$-submodules of $\ol{M}_I = M/\ol{IM}$ such that the morphisms $\ol{N}_J \to \ol{N}_I$, for $J \subset I$ in $\cP(k)$, are 
surjective for sufficiently small $I$. So $\{\ol{N}_I\}_{I \in \cP(k)}$ satisfies the condition in Definition~\ref{canon}.
\end{rmk} 
 So, if $M$ is canonical,  a basis of open $k$-submodules 
of $M$ consists 
of the closures $\ol{I\,M} = \pi_I^{-1}(\ol{M}_I)$ in $M$ of the submodules $I \,M$, for $I$ in a basis of open ideals of $k$, where $\pi_I : M \to \ol{M}_I$ denotes the canonical 
projection. For any $M$ in $\cL\cM_k$, the \emph{canonical topology} of $M$ is the topology  for which $\{\ol{IM}\}_{I \in \cP(k)}$ is a fundamental system of open neighborhoods of $0$. So, $M$ is separated 
in its canonical topology if and only if $\bigcap_{I \in \cP(k)} \ol{IM} = (0)$. 
\par
An object of $\cL\cM_{k}$ is not necessarily separated in its canonical topology as one sees in the case of the object 
$\Q_p$ of $\cL\cM_{\Z_p}$. The $p$-adic topology of $\Q_p$ is not $\Z_p$-canonical, since for any open ideal $I = p^N \Z_p$ of $\Z_p$, $I \Q_p = \Q_p$. 
\par
For any $M$ in $\cL\cM_k$ there is a natural exact sequence 
\beq \label{canexact}
0 \longrightarrow  \bigcap_{I \in \cP(k)} \ol{IM}  \longrightarrow  M \map{\gamma_M} \limPRO{I \in \cP(k)} M/\ol{IM} \longrightarrow 0 \;.
\eeq
\begin{lemma} \label{factexact} Let $M$ be an object of  $\cL\cM_k$, and let \eqref{canexact} be the corresponding exact sequence. Then \hfill
\ben 
\item $M$ is uniform if and only if its topology is weaker than the canonical topology.
\item $M$ is uniform if and only if $\bigcap_{I \in \cP(k)} \ol{IM} = (0)$.
\item $\limPRO{I \in \cP(k)} M/\ol{IM}$ is the maximal uniform quotient of $M$. 
\een
\end{lemma}
  \begin{proof}  The first part of the statement follows from Lemma~\ref{factexact0}.  The are parts are clear. 
  \end{proof}
  \begin{rmk} \label{naive2} It follows from Remark~\ref{OP} that if $k$  is a domain and satisfies {\bf OP} then, on any open ideal $J$ of $k$, the subspace topology of $J$ in $k$, the canonical topology of $J$ and the naive canonical topology of $J$ all coincide. 
   \end{rmk}
  \begin{prop} \label{naive3} Assume the object $M$ of $\cL\cM_k$ is equipped with the canonical topology and $k$ satisfies {\bf OP}.  Assume moreover that $M$ satisfies 
 \par \medskip
{\bf (TFM)} \emph{For any $I \in \cP(k)$, $M/\ol{IM}$ is a torsion-free  $k/I$-module.}
\par
\medskip \noindent
   Then,  condition {\bf OPM} holds for the $k$-linearly topologized $k$-module $M$. 
 \end{prop}
 \begin{proof} We may assume that $M = \what{N}$ as in \eqref{canform}, where $N$ is a $k$-module equipped with the naive canonical topology. By condition {\bf OP}, $ak$ is open in $k$ and $aN$ is open in $N$. We have to show that $aM$ is open in $M$. Equivalently, 
 we show that  
 $aM \supset \ol{aN}$ in $M$. Let $a x_\alpha$, for $x_\alpha \in N$ for any $\alpha \in A$, be a net converging to $y \in M$. We need to show that $\{x_\alpha\}_{\alpha \in A}$ converges in 
 $M$. So, it suffices to show that if the net $\{ax_\alpha\}_{\alpha \in A}$ is Cauchy in $N$, the net 
 $\{x_\alpha\}_{\alpha \in A}$ is Cauchy in $N$, as well. This follows from condition {\bf TFM}. In fact, for any $I \in \cP(k)$, $M/\ol{IM} \cong N/IN$ is a torsion-free  $k/I$-module. Then, let $I \in \cP(k)$ be such that $I \subsetneq a k$, and let $\alpha_0 \in A$ be such that $a(x_\alpha-x_\beta) \in IN$ for all $\alpha,\beta \geq \alpha_0$. 
 By condition {\bf TFM},  $x_\alpha-x_\beta \in IN$ for all $\alpha,\beta \geq \alpha_0$.  
 \end{proof}

\par \smallskip

\begin{lemma} \label{opencan} Let $i_N:N \to M$ be an open embedding in $\cL\cM_k$
with $M$ canonical. Then $N$ is canonical. 
\end{lemma}
 \begin{proof} This follows from the description of open sub-objects of $M$ in $\mathit 3$ of Lemma~\ref{exactsub} and from the Definition~\ref{canform}  of canonical modules.
 \end{proof} 
 \begin{defn} \label{uncondconv}
Let $M$ be an object of $\cL\cM_k$. A formal series $\sum_{\alpha \in A} m_\alpha$ of elements $m_\alpha \in M$, indexed by any set $A$,   \emph{converges  unconditionally} to $m \in M$ 
if the net $F \mapsto \sigma_F := \sum_{\alpha \in F} m_\alpha$, for $F$ a finite subset of $A$, converges to $m$.  
We then say that $m$ is the \emph{sum of the $A$-series $\sum_{\alpha \in F} m_\alpha$} and   write $m = \sum_{\alpha \in A} m_\alpha$.  
\end{defn}
\end{section}
\begin{section}{Complete tensor products} \label{tensors} \label{tensorprod}
 The category $\cL\cM_k$ and its full subcategory $\cL\cM^u_k$ admit various natural notions of complete tensor products. We are interested in two  of them. 
 \begin{defn} \label{bildef0}
 For 3 objects $M,N,P$ of $\cL\cM_k$, we denote by  $\Bil^c_k(M \times N,P)$
(resp. $\Bil^u_k(M \times N,P)$) the $k$-module of $k$-bilinear functions $f: M \times N \to P$ which are  continuous for the product topology of $M \times N$ 
(resp. uniformly  continuous for the product uniformity of $M \times N$). We denote by   $M \wt^c_k N$ (resp. $M \wt^u_k N$) the object of $\cL\cM_k$, if it exists, 
    which represents the functor $\cL\cM_k \to \cM od_k$ 
 $P \mapsto \Bil^c_k(M \times N,P)$ (resp.  $P \mapsto \Bil^u_k(M \times N,P)$).
  \end{defn}
  \begin{prop} 

  Let $M = \limIND{\alpha} M_\alpha$ and $N = \limIND{\beta} N_\beta$ be objects of $\cL\cM_k$, for filtered inductive systems $\{M_\alpha\}_\alpha$ and $\{N_\beta\}_\beta$ in 
$\cL\cM_k$. 
Then, for any $X$ in $\cL\cM_k$, there is a canonical isomorphism in $\cM od_k$
 \beq \begin{split}
\label{bilinearcont} 
\Bil_k^c(M \times N, X)& \iso  \limPRO{\alpha,\beta} \Bil_k^c(M_\alpha \times N_\beta,X) \\
\varphi  & \longmapsto (\varphi \circ j_{\alpha,\beta})_{\alpha,\beta}
\end{split}
\eeq
 where  $j_{\alpha,\beta}: M_\alpha \times N_\beta \to M \times N$ is the canonical morphism.  

 \end{prop}
 \begin{proof} The $k$-linear map of \eqref{bilinearcont} clearly exists and is injective. We have to show surjectivity. So, let $(\varphi_{\alpha,\beta})_{\alpha,\beta}$ be a compatible 
 system in $\limPRO{\alpha,\beta} \Bil_k^c(M_\alpha \times N_\beta,X)$. For any $(m,n) \in  M \times N = \limIND{\alpha.\beta} M_\alpha \times N_\beta$, let $(j_{\alpha,\beta}(m_\alpha,n_\beta))_{\alpha,\beta}$ 
 be a net indexed by $A \times B$ converging to $(m,n)$,  with $(m_\alpha,n_\beta) \in M_\alpha \times N_\beta$. Then the net  $(\varphi_{\alpha,\beta}(m_\alpha,n_\beta))_{\alpha,\beta}$ is a Cauchy net 
 in $X$. It suffices to check that for $(m,n) = (0,0)$ the previous net converges to $0 \in X$. Now, for any $U \in \cP(X)$ and any $(\alpha,\beta)$, there are $V_\alpha \in \cP(M_\alpha)$ and 
 $W_\beta \in \cP(N_\beta)$ such that $\varphi_{\alpha,\beta}(V_\alpha \times W_\beta) \subset U$. 
 By definition of inductive limits in $\cL\cM_k$, there exists $V \in \cP(M)$ (resp. $W \in \cP(N)$) such that, for any $(\alpha,\beta) \in A \times B$, $j_{\alpha,\beta}^{-1}(V \times W) \subset V_\alpha \times W_\beta$. 

 Since  $(j_{\alpha,\beta}(m_\alpha,n_\beta))_{\alpha,\beta}$  converges to $(0,0) \in M \times N$,  we may then assume that  $(m_\alpha,n_\beta) \in V_\alpha \times W_\beta$ for any $\alpha,\beta$. 
 It then follows that  $\varphi_{\alpha,\beta}(m_\alpha \times n_\beta) \in U$, for any $\alpha,\beta$. 

 Coming back to the case of any $(m,n) \in M \times N$, we define  $\varphi (m,n)$ as the limit of the net $(\varphi_{\alpha,\beta}(m_\alpha,n_\beta))_{\alpha,\beta}$. 
 The definition is good since it is  for $(m,n) = (0,0)$. It is clear that 
 $$
 \varphi \circ j_{\alpha,\beta} =  \varphi_{\alpha,\beta} \;\;\; \forall \;\;(\alpha,\beta) \in A \times B \;.
 $$

 \end{proof}
 
 \begin{lemma} \label{tenslemma} Let $M$ and $N$ be objects of $\cL\cM_k$. 
 \ben
 \item  
 $M \wt^c_k N$ exists in $\cL\cM_k$. More precisely, 
 $M \wt^c_k N$ is the completion of $M  \otimes_k N$ in the $k$-linear topology 
  with a fundamental set of open $k$-submodules given by the images $\Im(P  \otimes_k Q)$ in $M  \otimes_k N$, for $P$ (resp. $Q$) varying in the set of open submodules of $M$ (resp. $N$). So, 
 \beq \label{tenscdef}
M \wt^c_k N =  \limpro_{P,Q} \, \; (M  \otimes_k N)/  \Im (P  \otimes_k Q) 
\eeq
for $P$ (resp. $Q$) as before, where all the terms of the projective systems carry the discrete topology.  
\item  $M \wt^u_k N$ exists in $\cL\cM_k$. More precisely 
$M \wt^u_k N$ is the completion of $M  \otimes_k N$ in the $k$-linear topology 
 with a fundamental set of open $k$-submodules given by the images $\Im(P  \otimes_k N + M  \otimes_k Q)$ in $M  \otimes_k N$,
 for $P$ (resp. $Q$) varying in the set of open submodules of $M$ (resp. $N$).
 We have 
\beq \label{tensudef}
M \wt^u_k N =   \limpro_{P,Q} \, \; (M  \otimes_k N) /  \Im (P  \otimes_k N + M  \otimes_k Q) = \limpro_{P,Q} \, \; M/P \otimes_k N/Q  \;,
\eeq
for $P$ (resp. $Q$) as before, where all the terms of the projective systems carry the discrete topology and are uniform.  
A fundamental system  of open submodules of  $M \wt^u_k N$  consists of the closures in $M \wt^u_k N$   of  the $k$-submodules 
$\Im(P  \otimes_k N + M  \otimes_k Q) \subset M \otimes_kN$, for  $P, Q$ as before. 
 \een
\end{lemma}
\begin{proof} 
 Formula \eqref{tenscdef} is proven in essentially the same way as  \cite[Lemma 17.1]{schneider}.  Namely, we let $\cP(M)$ (resp. $\cP(N)$) be a basis of open $k$-submodules of 
 $M$ (resp. $N$). Then, for any object $T$ of $\cL\cM_k$  with a basis of open $k$-submodules $\cP(T)$, a continuous $k$-bilinear map  $\beta: M \times N \to T$ is in particular continuous 
 at $(0,0) \in M \times N$.  
Let  $\beta' : M \otimes_k N \to T$ be the $k$-linear  map corresponding to $\beta$. 
For any 
 $R \in \cP(T)$ there are $P \in \cP(M)$ and $Q \in \cP(N)$ such that $\beta (P \times Q) \subset R$ and therefore $\beta' (\Im(P \otimes_k Q)) \subset R$. We conclude that 
 $\beta'$ extends to a continuous $k$-linear map 
  $\gamma:= \what{\beta'} : M \wt^c_k N \to T$ such that 
  \beq \label{bilcont}
  \gamma(m \wt^c_k n) = \beta(m,n) \;.
  \eeq 
 Conversely, given the morphism $\gamma  : M \wt^c_k N \to T$ in $\cL\cM_k$,   the $k$-bilinear map  $\beta: M \times N \to T$ defined by \eqref{bilcont}
is such that  for any 
 $R \in \cP(T)$ there are $P \in \cP(M)$ and $Q \in \cP(N)$ such that $\beta (P \times Q) \subset R$. 
Moreover, for any $m \in M$ (resp. $n \in N$) the $k$-linear map $N \to T$ given by $y \mapsto \beta (m,y)$ (resp. $M \to T$ given by  $x \mapsto \beta (x,n)$) is continuous at $0$, hence is uniformly continuous. 
So,  
 for any fixed $(m,n) \in M \times N$ we can find $P_n \in \cP(M)$ and  $Q_m \in \cP(N)$ such that 
 $$\beta(m \times Q_m) , \beta (P_n,n)  , \beta (P_n,Q_m)  \subset   R \;.
 $$
 Then
 $$\beta ((m +P_n) \times (n + Q_m)) \subset  \beta (m,n) + \beta (\{m\} \times Q_m) + \beta (P_n \times \{n\}) + \beta (P_n \times Q_m) \subset   \beta (m,n)  +R \;.
  $$
  This proves that $\beta$ is continuous for the product topology of $M \times N$. 
  \par \smallskip
 We now pass to  \eqref{tensudef}~:
we prove the first equality in
that formula. Let $\beta: M \times N \to P$ be $k$-bilinear and uniformly continuous. Then, for any open submodule $W$ of $P$,
 we can find an  open submodule $U$ (resp $V$)   of $M$ (resp $N$) such that, for any $(x,y) \in M \times N$
 $$
 \beta(x+U,y+V) = \beta(x,y) + \beta(x,V) + \beta (U,y) + \beta (U \times V)  \subset \beta(x,y) + W \;.
 $$
This means that we must have $\beta(U \times N + M \times V) \subset W$. Conversely, if $\beta: M \times N \to P$ is $k$-bilinear and satisfies the latter condition 
the same calculation read backwards shows that $\beta$ is uniformly continuous.
The second equality in  \eqref{tensudef} follows from 
the canonical isomorphism 
$$
 M/P \otimes_k N/Q \iso (M  \otimes_k N) /  (P  \otimes_k N + M  \otimes_k Q)
$$
proven in 
\cite[II, \S 3, n.  6, Cor. 1 of Prop. 6, p. 60]{algebra}. 
\par
The fact that a fundamental system  of open submodules of   $M \wt^u_k N$ consists of the closures of  
$P  \otimes_k N + M  \otimes_k Q$, for  $P, Q$ in a fundamental system  of open submodules of $M,N$, respectively, is a general fact about completions. 
\end{proof}

\begin{rmk}
For any $M$ and $N$ in $\cL\cM_k$, we have a canonical  
 morphism in $\cL\cM_k$
$$
\Phi: M \wt^c_kN \longrightarrow M \wt^u_kN  
$$
such that, for any $X$ in $\cL\cM_k$, $\varphi \mapsto \varphi \circ \Phi$ is the natural inclusion
$$
\Bil_k^u(M \times N, X) \longrightarrow \Bil_k^c(M \times N, X)  \;.
$$ 
\end{rmk}

\begin{prop} \label{indmod} \hfill
\ben
\item
Let $M = \limIND{\alpha} M_\alpha$ and $N = \limIND{\beta} N_\beta$ be objects of $\cL\cM_k$, for inductive systems $\{M_\alpha\}_\alpha$ and $\{N_\beta\}_\beta$ in 
$\cL\cM_k$. 
Then
\beq
\label{billim0}
\limIND{\alpha} M_\alpha \, \wt^c_k \,  \limIND{\beta} N_\beta =  {\limIND{\alpha,\beta}} \, M_\alpha \wt^c_k N_\beta   \;.
\eeq
\item
Let $M = \limPRO{\alpha} M_\alpha$ and $N = \limPRO{\beta} N_\beta$ be objects of $\cL\cM_k$, for projective systems $\{M_\alpha\}_\alpha$ and $\{N_\beta\}_\beta$ in 
$\cL\cM_k$. Then
\beq
\label{billim1}
 \limPRO{\alpha} M_\alpha \, \wt^u_k \,  \limPRO{\beta} N_\beta   =   {\limPRO{\alpha,\beta}} \,  M_\alpha \wt^u_k N_\beta   \;.
\eeq
\item
If $M$ and $N$ are objects of $\cL\cM_k^u$,  then both $M \wt^c_kN$ and $M \wt^u_kN$ are uniform. 
\een 
\end{prop}
\begin{proof}   It suffices to prove that the functor $\Bil_k^c(M \times N, -)$ of \eqref{bildef0}
 is represented by  ${\limIND{\alpha,\beta}} \, M_\alpha \wt^c_k N_\beta$. So, we pick an object $X$ of $\cL\cM_k$ and consider 
  \eqref{bilinearcont}.  
The l.h.s.  of that equation in fact equals  $\hom_{\cL\cM_k}(M \wt_k^c N,X)$ while its r.h.s. equals 
$$ \limPRO{\alpha,\beta}  \hom_{\cL\cM_k}(M_\alpha \wt_k^c N_\beta,X) = \hom_{\cL\cM_k}(\limIND{\alpha,\beta} M_\alpha \wt_k^c N_\beta,X) \;.
  $$
  So, equation \eqref{billim0} follows. \par
To prove  equation \eqref{billim1} we write 
$$
M_\alpha = \limPRO{P_\alpha \in \cP(M_\alpha)} M_\alpha/P_\alpha   \;\;,\;\; N_\beta = \limPRO{Q_\beta \in \cP(N_\beta)} N_\beta/Q_\beta  \;.
$$
Then
\beq
\label{billim11}
\begin{split}
 \limPRO{\alpha} &\, M_\alpha \, \wt^u_k \,  \limPRO{\beta} \, N_\beta   =   ( \limPRO{\alpha}  \limPRO{P_\alpha \in \cP(M_\alpha)} M_\alpha/P_\alpha )\, \wt^u_k \,  (\limPRO{\beta}  \limPRO{Q_\beta \in \cP(N_\beta)} 
 N_\beta/Q_\beta) = \\ & \limPRO{\alpha}  \limPRO{P_\alpha \in \cP(M_\alpha)}  \limPRO{\beta}  \limPRO{Q_\beta \in \cP(N_\beta)} (M_\alpha/P_\alpha\, \otimes_k \, 
 N_\beta/Q_\beta)
 = {\limPRO{\alpha,\beta}} \,  M_\alpha \wt^u_k N_\beta   \;.
 \end{split}
\eeq
Point $\mathit 3$ is immediate. 
\par
  \end{proof}
 \end{section}
 \begin{section}{Pseudoconvexity} \label{Pseudoconvexity}  
   \begin{defn} \label{pseudocan} 
 An object $M$ of $\cL\cM_{k}$ is \emph{pseudobanach} or a  \emph{pseudobanach $k$-module}   if 
 there exists a family $\cG$ of open sub-objects of $M$   satisfying the following conditions
 \ben 
 \item Any $P \in \cG$ is equipped with the canonical topology.
 \item For any $P \in \cG$ and any open ideal $I$ of $k$, let $\ol{IP}$ be the closure of $IP$ in $P$. Then, the discrete $k/I$-module $P/\ol{IP}$ is flat. 
 \item
$M$ is the union of its open submodules $P$, for $P \in \cG$. 
 \een
   We call such a $\cG$ a $k$-\emph{gauge} (or simply a \emph{gauge} if there is no risk of confusion) of $M$.   
   The full subcategory of $\cL\cM_k$ consisting of pseudobanach objects will be denoted by $\cP\cB_k$.
   \end{defn}
   \begin{rmk} \label{basisopen} For any $M$ in $\cP\cB_k$ and any gauge $\cG$ for $M$, we have 
   $$
   M =\limpro_{P,I} M/\ol{IP}
   $$
   where $P \in \cG$ and $I$ describes the open ideals of $k$. Notice that the closure $\ol{IP}$ of $IP$ in $P$ is open in $M$. 
   So, while it is not required that $\cG$ should be a fundamental system of open $k$-submodules of $M$, 
   this is certainly the case for $\{ \ol{IP} \}_{P,I}$ for $P$ and $I$ as before. 
      \end{rmk}
      \begin{rmk} \label{commens} Notation as in Remark~\ref{basisopen}.
      For any $P,Q \in \cG$ we have that $P \cap Q$ is an open $k$-submodule of both $P$ and $Q$. It follows that there are open ideals $I,J$ of $k$ such that 
      $I P \subset Q$ and $J Q \subset P$. So, for any $P \in \cG$, $\{ \ol{IP} \}_I$, for $I$ an open ideal of $k$, is  a fundamental system of open $k$-submodules of $M$.  
       \end{rmk}
     \begin{rmk} \label{psudocank} Assume $k$ satisfies condition {\bf OP}.  Then follows from Remark~\ref{naive2}   that any open ideal $J$ of $k$  is a 
      pseudobanach $k$-module with gauge the set   $\cP(J)$ 
      of open ideals of $k$  contained in $J$.  
      \end{rmk}
   \begin{rmk} \label{psudodiscr} 
If $k$ has the discrete topology,  a pseudobanach $k$-module $M$ is simply a flat $k$-module equipped with the discrete topology. A gauge for such an  $M$ is $\cG = \{M\}$. So in this case $\cP\cB_k$ is the full subcategory of $\cL\cM_k$ consisting of flat $k$-modules equipped with the discrete topology and $-\wt_k^c - =   - \otimes_k -$.
 \end{rmk}
 \begin{rmk} \label{naive31} It follows from condition $\mathit 2$ of  Definition~\ref{pseudocan} that, for any pseudobanach $k$-module $M$, condition  {\bf (TFM)} of 
 Proposition~\ref{naive3} is satisfied for any $P$ in a gauge  $\cG$ for $M$. 
We then conclude from that proposition that,  if $k$ satisfies condition {\bf OP}, then, for any pseudobanach $k$-module $M$
and  for any $a \in k - \{0\}$, the
  map 
  $$M \longrightarrow M \;\;,\;\; x \longmapsto a x
 $$
 is open. 
 \end{rmk}
 \begin{rmk}
\label{TVS}
Let $K$ be a non-archimedean field and let $k = K^\circ$; then  $K$   is a pseudobanach object of $\cL\cR_{K^\circ}$ for the gauge $\{aK^\circ\}_{a \in K^\times}$. 
Let  $M$ be an object of $\cL\cM_k$ such that
for any $a \in k - \{0\}$, the map $m \mapsto am$ is a bijection. Then
     the scalar product $k \times M \to M$ extends uniquely to a structure of  $K$-vector space $K \times M \to M$. By Remark~\ref{naive31} 
     this is in fact a structure of topological $K$-vector space. 
  In this situation we will simply say that 
     the  object  $M$ of $\cL\cM_k$ \emph{is a topological $K$-vector space}. 

\end{rmk}
    \begin{defn}  \label{pseudobanach} 
    Let $K$ be a non-archimedean field and let $k = K^\circ$. A \emph{pseudobanach space over $K$} is any  object $M$ of $\cP\cB_k$ which is a $K$-vector space.
         We  view  the category $\cB an_K$ of pseudobanach spaces over $K$ as a full subcategory of $\cL\cM_{k}$.     
 \end{defn} 
 \begin{cor}
For a non-trivially valued $K$ the category $\cB an_K$  is  equivalent to the category 
 of $K$-Banach spaces and continuous maps of \cite{schneider}. For $K$ trivially valued  the category $\cB an_K$  is  equivalent to $\cM od_K$. 
 \end{cor}
\end{section}

\begin{section}{Topological rings}
\begin{defn} \label{topring0}
A   \emph{(topological) $k$-ring $A$} is a  $k$-linearly topologized topological $k$-module $A$ such that the product map
$$
\mu_A:  A \times A \longrightarrow A \;\;,\;\; (x,y) \longmapsto xy
$$
makes $A$ into a $k$-algebra (commutative with 1) and 
is continuous  \emph{for the product topology of $A \times A$}.  
We let  $\cL\cR_{k}$ be the category of complete topological $k$-rings and continuous $k$-algebra homomorphisms. We denote by $\cL\cR^u_{k}$ 
the full subcategory of $\cL\cR_{k}$ consisting of the objects  $A$ such that the scalar product 
$$
(\mu_A)_{| k \times A}:    k \times A \longrightarrow A \;\;,\;\; (\lambda,y) \longmapsto \lambda y
$$
is uniformly continuous. 
We define $\cR\cR_k$ as the full subcategory of $\cL\cR_k$ (and of $\cL\cR^u_k$) consisting of complete linearly topologized $k$-rings. 
\end{defn}
 If $A$ is an object 
of $\cR\cR_{k}$ then the product map is in fact uniformly continuous for the product uniformity of $A \times A$ but is not necessarily open. 
\begin{lemma} \label{colimrings} The categories  $\cL\cR_k$,  $\cL\cR^u_k$  and $\cR\cR_k$ admit both limits and colimits. The $\cR\cR_k$-limit of a projective system of elements of $\cR\cR_k$ 
coincides with its $\cL\cR^u_k$-limit and with its  $\cL\cR_k$-limit.  The $\cL\cR^u_k$-limit  of a projective system of elements of $\cL\cR^u_k$ 
coincides with its  $\cL\cR_k$-limit.  
\end{lemma}  
\begin{proof} 
 The case of limits in $\cR\cR_k$ follows from general nonsense. Namely, let  $(R_\alpha)_{\alpha \in A}$  be a  
 projective  system in $\cR\cR_k$. We then equip the projective limit 
 $R := \limPRO\alpha R_\alpha$ of $(R_\alpha)_{\alpha \in A}$ in $\cL\cM^u_k$ with a product map as follows. 
 For any $\alpha \in A$, the product map $\mu_{R_\alpha} : R_\alpha \times R_\alpha \to R_\alpha$ factors in this case 
 through a morphism $\mu_\alpha : R_\alpha \wt^u_k R_\alpha \to R_\alpha$. 
 We then have a projective system of morphisms in $\cL\cM^u_k$
 $$R \wt^u_k R =  (\limPRO\alpha R_\alpha) \wt^u_k  (\limPRO\alpha R_\alpha) = \limPRO\alpha (R_\alpha \wt^u_k R_\alpha) \map{\pi_\alpha \wt^u_k \pi_\alpha} R_\alpha \wt^u_k R_\alpha \map{\mu_\alpha }  R_\alpha  \longrightarrow R_\beta\;,
$$
for all $\alpha \geq \beta$, 
from which 
we obtain 
$$
\mu : R \wt^u_k R \longrightarrow \limPRO\beta R_\beta 
$$
and finally the product map
$$
\mu_R : R \times R \longrightarrow R \;.
$$
  The case of limits in $\cL\cR_k$ and $\cL\cR^u_k$  is similar to the one of $\cL\cM_k$ and $\cL\cM_k^u$ discussed in Lemma~\ref{limcolim} and will be omitted. 
  \par \medskip
  We now prove the existence of the colimit   in $\cL\cR_k$ of the inductive system  $(R_\alpha)_{\alpha \in A}$   in $\cL\cR_k$.  
For any $\alpha \in A$, the product map $\mu_{R_\alpha} : R_\alpha \times R_\alpha \to R_\alpha$ factors through a morphism $\mu_\alpha : R_\alpha \wt^c_k R_\alpha \to R_\alpha$.
We then equip the colimit  $R:= \limIND\alpha R_\alpha$  of the system $(R_\alpha)_{\alpha \in A}$ in $\cL\cM_k$ with the product map
obtained as follows.  From the system of morphisms 
$$
\mu_\alpha : R_\alpha \wt^c_k R_\alpha \longrightarrow  R_\alpha  \map{j_\alpha}  R
$$
and
$$
 R \wt^c_k R = (\limIND\alpha R_\alpha) \wt^c_k   (\limIND\alpha R_\alpha) \longrightarrow R_\alpha \wt^c_k R_\alpha
$$
we obtain 
$$
\mu:= \limIND\alpha \mu_\alpha : R \wt^c_k R = (\limIND\alpha R_\alpha) \wt^c_k   (\limIND\alpha R_\alpha) \longrightarrow  \limIND\alpha (R_\alpha \wt^c_k R_\alpha)  \longrightarrow R
$$
and finally the product map
$$
\mu_R : R \times R \longrightarrow R \;.
$$
Similarly, if the inductive system  $(R_\alpha)_{\alpha \in A}$  consists of objects of $\cL\cR^u_k$ (resp. $\cR\cR_k$), we
equip the colimit  $R^u:= {\limIND\alpha}^u R_\alpha$  of the system $(R_\alpha)_{\alpha \in A}$ in $\cL\cM^u_k$ with the product map
obtained as follows.  From the system of morphisms 
$$\mu_\alpha : R_\alpha \wt^c_k R_\alpha \longrightarrow  R_\alpha  \map{j_\alpha}  R \;,
$$
where $R_\alpha \wt^c_k R_\alpha$ is uniform by comma $\mathit 3$ of Proposition~\ref{indmod}, 
and
$$
 R^u \wt^c_k R^u = ({\limIND\alpha}^u R_\alpha) \wt^c_k   ({\limIND\alpha}^u R_\alpha) \longrightarrow R_\alpha \wt^c_k R_\alpha
$$
we obtain 
$$
\mu_{R^u} := {\limIND\alpha}^u \mu_\alpha : R^u \wt^c_k R^u = ({\limIND\alpha}^u R_\alpha) \wt^c_k   ({\limIND\alpha}^u R_\alpha)  \longrightarrow  {\limIND\alpha}^u (R_\alpha \wt^c_k R_\alpha)   \longrightarrow R^u
$$
and finally the continuous product map
$$
\mu_{R^u} : R^u \times R^u \longrightarrow R^u \;.
$$

\end{proof} 
\begin{rmk}
Although not logically necessary, we prefer to 
give an explicit description of  the product map of $R$.
As in Lemma~\ref{limcolim},
 we first consider $R' = \limind_{\alpha \in A} R_\alpha^\for$ in $\cR ings$ and let $j_\alpha: R^\for_\alpha \to R'$ 
be the natural morphisms.   
We then give to $R'$ the finest $k$-linear  topology  such that all maps $j_\alpha: R_\alpha \to R'$ are continuous. So, a basis of open $k$-submodules in $R'$ consists of 
the $k$-submodules $U$ of $R'$  such that
$j_\alpha^{-1}(U)$ is an open $k$-submodule $J_\alpha$ of $R_\alpha$, for any $\alpha \in A$. 
Then $\limIND{\alpha \in A} R_\alpha$ is represented by the completion $R$ of $R'$ in that 
topology, equipped with the natural morphisms $i_\alpha : R_\alpha \to  R$ deduced from the $j_\alpha$'s. 
Let $r = (r_\alpha)_{\alpha \in A}, s = (s_\alpha)_{\alpha \in A} \in R$. Then, for any 
open $k$-submodule $U$ of $R$ as before there is an index $\alpha_0 \in A$  
such that  for any $\alpha \geq \alpha_0$, $r_{\alpha} - r_{\alpha_0}, s_{\alpha} - s_{\alpha_0}  \in J_\alpha$.
So, 
$$r_{\alpha} s_{\alpha} - r_{\alpha_0}s_{\alpha_0} = r_{\alpha}(s_{\alpha} - s_{\alpha_0}) + (r_{\alpha} - r_{\alpha_0})s_{\alpha_0} \in J_\alpha \;.
$$
This shows that  $(r_\alpha s_\alpha)_\alpha  \in R$ so that we get a product map
\beqa
\begin{array}{ccccccc}
& R = \limIND{\alpha \in A} R_\alpha &  \times  & R = \limIND{\alpha \in A} R_\alpha & \longrightarrow & R =  \limIND{\alpha \in A} R_\alpha& \\
 &&&&&&
 \\
  &   (r=(r_\alpha)_{\alpha \in A} & ,  &s= (s_\alpha)_\alpha  )& \longmapsto& rs =(r_\alpha s_\alpha)_\alpha
\end{array}
\eeqa
continuous for the product topology of $R \times R$. It is clear that $R$ is in fact the colimit of the inductive system  $(R_\alpha)_{\alpha \in A}$   in $\cL\cR_k$.  
  \par
Assume now   $(R_\alpha)_{\alpha \in A}$ 
is an  inductive system  in $\cL\cR^u_k$. The previous construction gives the inductive limit  of $(R_\alpha)_{\alpha \in A}$ in $\cL\cR_k$. 
To construct explicitly 
the inductive limit  of $(R_\alpha)_{\alpha \in A}$ in $\cL\cR^u_k$ we repeat the construction of $R'$ but endow it with the finest $k$-linear topology \emph{weaker than the canonical 
topology}
such that all $j_\alpha$ are continuous. Then $R^u$ is the completion of $R'$ in that topology, and  the existence of a product map
$$
\mu_{R^u} : R^u \times R^u \longrightarrow R^u 
$$  
follows. \par
Finally, let $(R_\alpha)_{\alpha \in A}$ 
is an  inductive system  in $\cR\cR_k$.
To explicitly  construct 
the inductive limit  of $(R_\alpha)_{\alpha \in A}$ in $\cR\cR_k$ we repeat the construction of $R'$ but endow it with the finest \emph{linear} topology 
such that all $j_\alpha$ are continuous. A basis of open ideals of $R'$ then consists of the ideals $U$ such that $j_\alpha^{-1}(U) = J_\alpha$ is an open ideal of $R_\alpha$, for any $\alpha \in A$. 
 We prove as before that the completion $R$ of $R'$ in the latter topology is an object of $\cR\cR_k$ and that it represents the inductive limit  of $(R_\alpha)_{\alpha \in A}$ in $\cR\cR_k$.
\end{rmk}
\begin{notation} \label{indlimdef}  If $\{R_\alpha\}_\alpha$ is an inductive system in 
 $\cR\cR_k$ (resp. $\cL\cR^u_k$, resp. $\cL\cR_k$), we denote by $\limind^\cR_\alpha R_\alpha$ (resp. $\limind^u_\alpha R_\alpha$, resp. $\limind^\cL_\alpha R_\alpha$) its inductive limit in  $\cR\cR_k$
(resp. $\cL\cR^u_k$, resp. $\cL\cR_k$). 
\end{notation}
We introduce in the case of rings a notation analog to \eqref{cohnhd}.
\begin{notation} \label{cohnhdrings}  
Let $(R_\alpha)_{\alpha \in A}$ be 
an inductive system  in $\cL\cR_k$ with transition morphisms $j_{\alpha,\beta}:R_\alpha \to R_\beta$ for $\alpha \leq \beta$. 
For any $\alpha \in A$ let $\cP(R_\alpha)$ denote the set of open ideals of $R_\alpha$. Then 
\emph{a coherent system  of open ideals of $(R_\alpha)_{\alpha \in A}$} is a system $\sJ := (J_\alpha)_{\alpha \in A}$ such that for any $\alpha \leq \beta$ in $A$, 
$j_{\alpha,\beta}^{-1}(J_\beta) = J_\alpha$. The set $\sC((R_\alpha)_{\alpha \in A})$ of coherent systems of open ideals of $(R_\alpha)_{\alpha \in A}$ forms a filter
of $k$-submodules of $\prod_{\alpha \in A} R_\alpha^\for$. 
\end{notation}
\begin{lemma} \label{expldirlimrings}   We use the notation of \eqref{cohnhdrings}.
\hfill \ben
\item Let $(R_\alpha)_{\alpha \in A})$ be 
an inductive system  in $\cL\cR_k$. Then
\beq  \label{expldirlimrings1}
{\limIND\alpha}^\cL R_\alpha = \limPRO{\sJ \in \sC((R_\alpha)_{\alpha \in A})} {\limIND\alpha} R_\alpha/J_\alpha 
\eeq
where the inductive limit 
$$
{\limIND\alpha} R_\alpha/J_\alpha
$$
is taken in the category $\cM od_k$. 
\item Let 
$(R_\alpha)_{\alpha \in A}$ is an inductive system  in $\cL\cR^u_k$. Then 
\beq  \label{expldirlimrings2}
{\limIND\alpha}^u R_\alpha = \limPRO{I \in \cP(k)} {\limIND{\alpha \in A}}^u R_\alpha/ \ol{I R_\alpha}  
\eeq
where the inductive limit 
$$
 {\limIND{\alpha \in A}}^u R_\alpha/ \ol{I R_\alpha}
$$
is taken in the category $\cL\cR^u_{k/I}$. It coincides as a topological $k/I$-module with  the colimit of the same inductive system taken in the category 
$\cL\cM^u_{k/I}$ and also in the category $\cL\cM_{k/I} = \cL\cM^u_{k/I}$
since $k/I$ is discrete. 
\item Let 
$(R_\alpha)_{\alpha \in A}$ is an inductive system  in $\cR\cR_k$. Then 
\beq  \label{expldirlimrings3}
{\limIND\alpha}^\cR R_\alpha = \limPRO{I \in \cP(k)} \limIND{\alpha \in A}^\cR R_\alpha/ \ol{I R_\alpha} \;, 
\eeq
where the inductive limit 
$$
\limIND{\alpha \in A}^\cR R_\alpha/ \ol{I R_\alpha} 
$$
is taken in the category $\cR\cR_{k/I}$.
\een
 \end{lemma}
\begin{proof} Clear.
\end{proof}


   By general nonsense, for any  inductive system $\{R_\alpha\}_\alpha$ in 
 $\cR\cR_k$,  there is a canonical morphism in $\cL\cR_k$ 
 \beq  \label{indlim0} 
 T :  {\limind_\alpha}^\cL R_\alpha \longrightarrow {\limind_\alpha}^\cR \, R_\alpha \;.
 \eeq
For any  object $R$ in  $\cR\cR_k$  the map 
\beq  \label{indlim11} \begin{split}
 \hom_{\cR\cR_k}({\limind_\alpha}^\cR \,R_\alpha,R)& \longrightarrow  \hom_{\cL\cR_k}({\limind_\alpha}^\cL \, R_\alpha,R)   \\
\varphi & \longmapsto  \varphi \circ   T 
\end{split}
\eeq
 is in fact an isomorphism, since both source and target equal $\limPRO\alpha  \hom_{\cR\cR_k}(R_\alpha,R)$. 
 \begin{lemma}  \label{indlim1} \hfill
 \ben
 \item
 There exists a $k$-linear functor 
 $$\cT : \cI\cR\cR_k \longrightarrow \cR\cR_k \;\;\;\mbox{such that} \;\;\;   \cT ( {\limind_\alpha}^\cL R_\alpha) = {\limind_\alpha}^\cR \, R_\alpha  \;,
 $$
for any  inductive system $\{R_\alpha\}_\alpha$ in 
 $\cR\cR_k$.    \item
 Let 
 \beq \label{natural} 
 \iota_{\cI\cR\cR_k} : \cI\cR\cR_k \hookrightarrow \cL\cR_k \;\;\mbox{and}\;\;  \iota_{\cR\cR_k} : \cR\cR_k \hookrightarrow \cL\cR_k
 \eeq 
 be the natural inclusions of full subcategories. 
  There exists a natural transformation of functors $\cI\cR\cR_k \to  \cL\cR_k$ 
 \beq \label{natutransf}
 \cS: \iota_{\cI\cR\cR_k} \longrightarrow   \iota_{\cR\cR_k} \circ \cT
 \eeq
 such that, for any object $\limind_\alpha^\cL R_\alpha$ of $\cI\cR\cR_k$
  \beq \label{natutransf1}
  \cS({\limind_\alpha}^\cL R_\alpha) : {\limind_\alpha}^\cL R_\alpha \longrightarrow {\limind_\alpha}^\cR \, R_\alpha
 \eeq
 coincides with the morphism $T$ of \eqref{indlim0}, or,  equivalently,  is the image of  
 $\id_R$, for $R := \limIND\alpha R_\alpha$, via the identification of \eqref{indlim11}.
 \een
 \end{lemma}
 \begin{proof}   The fact that the  correspondence of objects
 $$\cT : {\limind_\alpha}^\cL R_\alpha \longmapsto {\limind_\alpha}^\cR R_\alpha
 $$ 
is well-defined follows from \eqref{indlim11}. The  fact that $\cT$ extends to a functor, is general nonsense, and completes the proof of $\mathit 1$. Part $\mathit 2$ is self-explanatory.
 \end{proof}
 \begin{exa} \label{nonlinexa} \hfill
 \ben
 \item
 A typical example of an object of $\cL\cR_{\Z_p}$, but not of  $\cR\cR_{\Z_p}$, is any non-archimedean non-trivially valued field extension $K$ of $(\Q_p,v_p)$. 
 This $K$ is also an object of $\cL\cR_{K^\circ}$ but not of 
 $\cR\cR_{K^\circ}$.  The same situation occurs for any commutative $K$-Banach algebra. 
 \item Let $k = \Z_p$ and let $\Z_p\{x\}$ be the $p$-adic completion of $\Z_p[x]$. So, $\Z_p\{x\}$  is an object of $\cR\cR_k$. Let $F:  \Z_p\{x\}  \to  \Z_p\{x\} $ be the $\cR\cR_k$-morphism 
 such that $F(x) = px$. Consider the inductive system 
 $$(\Z_p\{x\},F) := \Z_p\{x\} \map{F} \Z_p\{x\} \map{F} \dots
 $$
 in $\cR\cR_k$. Then 
 $$
 {\limind}^\cR (\Z_p\{x\},F) =  {\limind}^u (\Z_p\{x\},F)  = \Z_p  \;\;\mbox{while} \;\; {\limind}^\cL (\Z_p\{x\},F)  = \Q_p\{x\} \;.
 $$
 \een
 \end{exa}   
\end{section}
 \begin{section}{Tensor product of rings}
 \label{prodalg}
For two objects $A,B$ of $\cL\cR^u_k$, the $k$-module
$A   \wt^u_k B$ is naturally an object of $\cL\cR^u_k$ with product 
\beq \label{prodalg1} \begin{split}
A  \wt^u_k B  \, \times  \, A  \wt^u_k B &\longrightarrow A   \wt^u_k B
\\
(a_1  \wt^u_k  b_1   ,  a_2  \wt^u_k  b_2) & \longmapsto  a_1 a_2  \wt^u_k  b_1 b_2 \; .
\end{split}
\eeq
If, in particular, $A,B$ are objects of $\cR\cR_k$, so is $A   \wt^u_k B$.
Moreover, for any object $R$ of $\cL\cR^u_k$ the product $R \times R \to R$ factors through a morphism 
$$\mu_R: R \wt^u_k R \to R
$$ in  $\cL\cM^u_k$.  This holds in particular if $R$ is in  $\cR\cR_k$. 
Similarly, for two objects $A,B$ of $\cL\cR_k$, the $k$-module 
$A   \wt^c_k B$ is naturally an object of $\cL\cR_k$ with a similar formula for the product.  For any object $R$ of $\cL\cR_k$ the product $R \times R \to R$ factors through a morphism 
$$\mu_R: R \wt^c_k R \to R
$$ in  $\cL\cM_k$.

\begin{defn} \label{produnif} For objets $A,B,C$ of  $\cL\cR_k$   we denote by   $\Ril_k^c(A \times B, C)$ (resp. $\Ril_k^u(A \times B, C)$) the $k$-submodule of 
  $\Bil_k^c(A \times B, C)$ (resp. $\Bil_k^u(A \times B, C)$) consisting of functions $(a,b) \mapsto \varphi (a,b)$ such that
 for any $a \in A$ (resp. $b \in B$)
the map 
$b \mapsto \varphi (1,b)$ (resp. $a \mapsto \varphi (a,1)$) is a morphism in $\cL\cR_k$ and
$$
 \varphi (a,b) =  \varphi (a,1)  \varphi (1,b) \;.
$$
\end{defn}
For any objets $A,B$ of  $\cL\cR_k$, we consider the functors
\beq \label{Rilcdef}
\Ril_k^c(A \times B, -) : \cL\cR_k \longrightarrow \cM od_k  \;,
\eeq 
\beq \label{Riludef}
\Ril_k^u(A \times B, -)  : \cL\cR_k \longrightarrow \cM od_k  \;.
\eeq

\begin{lemma} \label{tensrep} \hfill
\ben
\item
For any $A,B$ in $\cL\cR_k$, the  functor $\Ril_k^c(A \times B, -)$ of \eqref{Rilcdef}  is represented by the object $A \wt_k^c B$ of  $\cL\cR_k$.
\item
For any $A,B$ in $\cL\cR_k$, the  functor $\Ril_k^u(A \times B, -)$ of \eqref{Riludef}  is represented by the object $A \wt_k^u B$ of  $\cL\cR_k$. 
\item If $A$ and $B$ are objects of $\cL\cR^u_k$ (resp. $\cR\cR_k$), then $A \wt_k^u B$ is an object  of  $\cL\cR^u_k$ (resp. of $\cR\cR_k$).
\item For any $A,B$ in $\cL\cR_k$,   there is a canonical  morphism 
\beq
\label{bilunif20}
\Phi: A \wt_k^c B \longrightarrow A \wt_k^u B  
\eeq
such that, for any $R$ in $\cL\cR^u_k$, $\varphi \mapsto \varphi \circ \Phi$ is the natural inclusion
$$
\Ril_k^u(A \times B, R) \longrightarrow \Ril_k^c(A \times B, R)  \;.
$$ 
\een
\end{lemma}
\begin{proof} The first and the second assertions of the lemma are proved similarly; we only prove the second.  
 
We recall that (for any $A$ and $B$ in $\cL\cM_k$) $A \wt_k^u B$ is the separated completion of $A \otimes_k B$ (meaning of $A^\for \otimes_{k^\for} B^\for$) equipped with a suitable $k$-linear topology. 
For any $X$ in $\cL\cM_k$ we have an isomorphism of $k^\for$-modules
\beq \label{bilrep} \begin{split}
\Bil_k^u(A \times B,  X) &\iso \hom_{\cL\cM_k} (A \wt_k^u B,X) \\
\varphi &\longmapsto (\Phi:a \otimes_k b \mapsto \varphi(a,b))
\;.
\end{split}
\eeq
In the present situation,   we also have a ring structure on  $A \wt_k^u B$ and maps $j_1:A \to A \wt_k^u B$ (resp. $j_2: B \to A \wt_k^u B$) which extend by continuity 
$a \mapsto a \otimes_k 1$ (resp. $b  \mapsto 1 \otimes_k b$) and are morphisms in  $\cL\cR_k$. Then the identification \eqref{bilrep} is characterized by the 
condition that $\Phi \circ j_1 : A \to X$ is $a \mapsto \varphi (a,1)$ while $\Phi \circ j_2 : B \to X$ is $b \mapsto \varphi (1,b)$. 
If $X$ is an object of $\cL\cR_k$ 
the identification \eqref{bilrep} 
restricts to an identification 
$$
\Ril_k^u(A \times B,  X) = \hom_{\cL\cR^u_k} (A \wt_k^u B,X) \;.
$$
The other assertions of the statement are clear. 

\end{proof} 
 
\begin{prop} \label{indrep} \hfill
\ben
\item
Let  $\{A_\alpha\}_\alpha$ and $\{B_\beta\}_\beta$ be   inductive systems  in 
$\cL\cR_k$. 
Then
\beq
\label{bilunif0}
{\limIND\alpha}^\cL A_\alpha \, \wt^c_k \,  {\limIND\beta}^\cL B_\beta =  {\limIND{\alpha,\beta}}^\cL  (A_\alpha \wt^c_k B_\beta )
\;.
\eeq 
\item  
Let $\{A_\alpha\}_\alpha$ and $\{B_\beta\}_\beta$ be  projective systems  in 
$\cL\cR_k$. 
Then
\beq
\label{bilunif1}
{\limPRO\alpha} A_\alpha \, \wt^u_k \,  {\limPRO\beta} B_\beta =  {\limPRO{\alpha,\beta}} A_\alpha \wt^u_k B_\beta  
\;.
\eeq 
\een
\end{prop}
\begin{proof} The two  statements follow from \eqref{billim0} and  \eqref{billim1}, respectively.   
 \end{proof}
 \end{section}
 \begin{section}{Base change} \label{basechangesec}
\begin{lemma} \label{basechange}
Let $A$ be any object of  $\cR\cR_k$.  
\ben
\item
Let $M$ be an object of  $\cL\cM_k$ (resp. of  $\cL\cM^u_k$). 
The map 
$$A \times (A \otimes_k M) \to A \otimes_k M \;\;,\;\; (a,b \otimes m) \mapsto ab \otimes m \;,
$$  
extends to  
an $A$-bilinear map  
$$A \times (A \wt^c_k M) \to A \wt^c_k M
$$
(resp. 
$$A \times (A \wt^u_k M) \to A \wt^u_k M\;\;\mbox{)}
$$
(resp. uniformly) continuous for the product topology (resp. uniformity) of $A \times (A \wt^c_k M)$ (resp. $A \times (A \wt^u_k M)$)
which makes 
$A \wt_k^c M$ (resp. $A \wt_k^u M$)  into an $A$-linearly topologized (resp. uniform) 
separated and complete $A$-module. 
The correspondence 
\beq  \label{scalext}
\begin{split}
(-)^c_A : \cL\cM_k &\longrightarrow \cL\cM_A
\\
M &\longmapsto  (M)^c_A := A \wt_k^c M
\end{split}
\eeq
(resp.
\beq  \label{scalextu1}
\begin{split}
(-)^u_{A} : \cL\cM^u_k &\longrightarrow \cL\cM^u_A
\\
M &\longmapsto  (M)^u_{A} := A \wt_k^u M\;\;\mbox{)}
\end{split}
\eeq
is part of an additive functor which we call \emph{continuous} (resp. \emph{uniform}) \emph{extension of scalars} by $A$.
The functor $(-)^c_{A}$ (resp. $(-)^u_{A}$) commutes with inductive (resp. projective)  limits in $\cL\cM_k$ and $\cL\cM_A$ (resp. in $\cL\cM^u_k$ and $\cL\cM^u_A$). 
The functor $(-)^c_{A}$  (resp. $(-)^u_{A}$)  is left-adjoint to the natural inclusion functor $\cL\cM_A \hookrightarrow  \cL\cM_k$ (resp. $\cL\cM^u_A \hookrightarrow  \cL\cM^u_k$). Namely,
for any $M$ in $\cL\cM_k$ and $N$ in $\cL\cM_A$ (resp. for any $M$ in $\cL\cM^u_k$ and $N$ in $\cL\cM^u_A$)
\beq \label{adjc}
\hom_{\cL\cM_k}(M,N) = \hom_{\cL\cM_A}((M)^c_{A},N)
\eeq
(resp. 
\beq \label{adju}
  \hom_{\cL\cM^u_k}(M,N) = \hom_{\cL\cM^u_A}((M)^u_{A},N) \; \mbox{)}\;.
\eeq 
\item For any $B$ in $\cL\cR_k$ (resp. in $\cL\cR^u_k$, resp. in $\cR\cR_k$), the  canonical morphism 
$A \to A \wt_k^c B$, $a \mapsto a \otimes 1$, makes $A \wt_k^c B$ into an $A$-linearly (resp. a uniform $A$-linearly, resp. a linearly) topologized 
separated and complete $A$-ring. Moreover
\beq \label{limmon1c}
\begin{split}
(-)^c_{A} : \cL\cR_k &\longrightarrow \cL\cR_A \\
B & \longmapsto    (B)^c_{A} := A \wt_k^c B
\end{split}
\eeq
(resp.
\beq  \label{scalextu2}
\begin{split}
(-)^u_{A} : \cL\cR^u_k &\longrightarrow \cL\cR^u_A
\\
B &\longmapsto  (B)^u_{A} := A \wt_k^u B \;\; ,
\end{split}
\eeq
resp.
\beq  \label{scalextu3}
\begin{split}
(-)^u_{A} : \cR\cR_k &\longrightarrow \cR\cR_A
\\
B &\longmapsto  (B)^u_{A} := A \wt_k^u B \;\;\mbox{)}
\end{split}
\eeq

is an additive functor, commuting with inductive limits in $\cL\cR_k$ and $\cL\cR_A$ (resp. with projective limits in $\cL\cR^u_k$ and $\cL\cR^u_A$, 
resp. with projective limits in $\cR\cR_k$ and $\cR\cR_A$), which we call \emph{continuous} (resp. \emph{uniform}, resp. \emph{uniform}) \emph{base-change} by $A$.  
The functor $(-)^c_{A}$  (resp. $(-)^u_{A}$, resp. $(-)^u_{A}$) is left-adjoint to the natural inclusion functor $\cL\cR_A \hookrightarrow  \cL\cR_k$ (resp. $\cL\cR^u_A \hookrightarrow  \cL\cR^u_k$, resp. $\cR\cR_A \hookrightarrow  \cR\cR_k$). Namely,
for any $R$ in $\cL\cR_k$ and $S$ in $\cL\cR_A$ (resp. for any $R$ in $\cL\cR^u_k$ and $S$ in $\cL\cR^u_A$, resp. for any $R$ in $\cR\cR_k$ and $S$ in $\cR\cR_A$)
\beq
\hom_{\cL\cR_k}(R,S) = \hom_{\cL\cR_A}((R)^c_{A},S) 
\eeq
(resp. 
\beq
\hom_{\cL\cR^u_k}(R,S) = \hom_{\cL\cR^u_A}((R)^u_{A},S)   \;,
\eeq 
resp. 
\beq
\hom_{\cR\cR_k}(R,S) = \hom_{\cR\cR_A}((R)^u_{A},S) \;\;\mbox{)} \;.
\eeq 
\een
In particular, $(-)^u_{A} :  \cR\cR_k \to  \cR\cR_A$ commutes with direct limits. 
\end{lemma}
\begin{proof} 
We prove the adjunction property for $(-)^u_{A}$. 
Let $M$ (resp. $N$) be an object of $\cL\cM^u_k$ (resp. $\cL\cM^u_A$). 
Then  from \eqref{modhom} we get 
\beq \label{modhom3}
\hom_{\cL\cM^u_A}(A \wt^u_k M ,N) = \limpro_{Q \in \cP(N)}  \limind_{P \in \cP(M)} \hom_A(M/P  \otimes_k A,N/Q)
\eeq
where $\hom_A(M/P  \otimes_k A,N/Q) = \hom_{A/J}(M/P  \otimes_k A/J,N/Q)$ for any $J \in \cP(A)$ such that $JN \subset Q$ and $(J \cap k)M \subset P$. 
If $J \cap k \supset I \in \cP(k)$ so that $IM \subset P$, the latter equals $\hom_{k/I}(M/P,N/Q)$ and we conclude by  \eqref{modhom}.
\par We now prove the adjunction property for $(-)^c_{A}$. 
Let $M$ (resp. $N$) be an object of $\cL\cM_k$ (resp. $\cL\cM_A$). Then 
\beq \label{modhom5}
\begin{split}
& \hom_{\cL\cM_A} (A \wt^c_k M ,N)   = \{ \varphi \in \hom_{\cL\cM_k}(A \wt^c_k M ,N) | \, \varphi \, \mbox{is $A$-linear}\, \} 
= \\ \{ \varphi \in & \Bil_k^c(A \times M ,N) | \, \varphi\, \mbox{is $A$-linear in the first variable}\, \} = \hom_{\cL\cM_k}(M ,N) \;.
\end{split}
\eeq 
\end{proof}
 
\begin{defn} \label{tffdef} 
Let $f: A \to B$ be a morphism in $\cR\cR_k$. We say that $f$  is \emph{pro-flat} or that $B$ is \emph{pro-flat}  over $A$ if, for any open ideal $J$ in $B$, $B/J$ is a flat $A/f^{-1}(J)$-module. 
 \end{defn}

 \begin{lemma}  \label{basechange11} Let $M$ and $N$ be objects of $\cL\cM_k$ endowed with the $k$-canonical topology (so, in particular,  $M$ and $N$ are uniform) and let $A$ be an object of $\cR\cR_k$.
\ben
\item $M \wt_k^u N$  carries the $k$-canonical topology. If $k$   satisfies {\bf OPW} then  $M \wt_k^c N  = M \wt_k^u N$.
\item
 The morphism $(M)^c_{A} \to (M)^u_{A}$ is an isomorphism of $\cL\cM^u_A$ and both objects  carry the $A$-canonical topology.  
\item Let $I$ (resp. $J$)  be an open ideal  of $k$ (resp. $A$) such that $IA \subset J$.  If $M/\ol{IM}$ is a flat $k/I$-module, 
$(M)^u_{A}/\ol{J(M)^u_{A}} = (M)^c_{A}/\ol{J(M)^c_{A}}$ is a flat $A/J$-module. 
\item Assume $A$ is pro-flat over $k$ and let  $L$ be an open sub-object of $M$. Then $(L)^u_{A}$ is  an open  sub-object of $(M)^u_{A}$.  
\item If $A$ is pro-flat over $k$,  for any open sub-object $L$ of $M$, $(L)^u_{A} = (L)^c_{A}$  is  an open  sub-object of $(M)^u_{A} = (M)^c_{A}$ and all these objects carry the 
$A$-canonical topology. 
\een
\end{lemma} 
\begin{proof} \hfill
\par 
$\mathit 1.$  According to \eqref{tensudef} it suffices to show  that
$$
C := \Ker (M \wt_k^u N   \to  M/\ol{IM} \otimes_k N/\ol{IN}) 
$$ is  the closure of $I (M \wt_k^u N)$ in $M \wt_k^u N$.
By \cite[II, \S 3, n.  6, Cor. 1 of Prop. 6]{algebra}, the kernel of 
the morphism
$$
M^\for \otimes_{k^\for}  N^\for \longrightarrow M^\for/I M^\for \otimes_{k^\for} N^\for/I N^\for
$$
is 
$$
C' := M^\for \otimes_{k^\for} I N^\for + I M^\for \otimes_{k^\for} N^\for = I(M^\for \otimes_{k^\for}  N^\for)
\; .
$$
So, $C$ coincides with the
 closure of $C'$ in $M\wt^u_k N$   and therefore also with  
the closure of 
 $I (M \wt_k^u N)$ in $M\wt^u_k N$, as claimed.  
 \par
 If $k$  satisfies {\bf OPW}, then the ideals of the form $IJ$, for $I,J \in \cP(k)$ are open. So the system of $k$-submodules $IM \otimes_kJN$, for $I,J \in \cP(k)$, is cofinal with the system 
 $IM \otimes_k N$, for $I \in \cP(k)$. So, $M\wt^c_k N$ coincides with $M\wt^u_k N$. 
\par $\mathit 2.$ $(M)^c_{A}$ is the completion of $A \otimes_k M$ in the $k$-linear topology with fundamental system of open $k$-submodules 
given by the family of the $\Im(J   \otimes I\,M)$, where $I$ (resp. $J$) runs over a fundamental system of open ideals of $k$ (resp. $A$). This equals the
fundamental system of open $A$-submodules $\{\Im((J + I \,A) \otimes M)\}_{I,J}$ for $I,J$ as before.  A cofinal system $\cF$  is given by the condition $I \, A \subset J$, that is by 
$$\cF = \{ \Im((J + I \,A) \otimes M) \,\}_{I \subset J} = \{\Im(J \otimes_k M)\}_J =  \{J \,(A \otimes_k M)\}_J 
\;,
$$
so that the topology obtained on $(M)^c_{A}$ is indeed the $A$-canonical one.  Similarly for $(M)^u_{A}$.
\par $\mathit 3.$ 
Let $J$ be 
any open ideal of $A$, and let $I$ be an open ideal of $k$ such that $IA \subset J$. Then 
$$(M)^c_{A}/\ol{J(M)^c_{A}} = (M)^c_{A}/ \ol{J \otimes_k M} =  (M)^c_{A}/ \ol{J \otimes_k M + A \otimes_k IM}   = A/J \otimes_{k/I} M/\ol{IM} 
$$ 
is a flat $A/J$-module. Similarly for $(M)^u_{A}/\ol{J(M)^u_{A}}$. 
\par $\mathit 4.$ The open sub-object $L$ of $M$ may be described as a projective system of submodules $\ol{L}_I \subset M/\ol{IM}$, for $I \in \cP(k)$, such that the maps 
$\ol{L}_J \to \ol{L}_I$, for $J\subset I$ in $\cP(k)$, are surjective for sufficiently small $I$. This properties are inherited by the projective system  $\ol{L}_I \otimes_{k/I} A/H$ for any 
$H \in \cP(A)$, where $\ol{L}_I \otimes_{k/I} A/H \to M/\ol{IM} \otimes_{k/I} A/H$  is injective since $A$ is pro-flat. The projective system  $\ol{L}_I \otimes_{k/I} A/H$ then  defines an open sub-object of 
$(M)^u_{A}$. Since $(-)^u_{A}$ commutes with projective limits, the open sub-object of $(M)^u_{A}$ which we get is 
$(L)^u_{A}$.
\par $\mathit 5.$  This is simply a summary of what we have proven, taking into account Remark~\ref{opencanon}.
\end{proof}

\begin{prop} \label{tenspsudcan}  We assume that $k$   satisfies condition {\bf OPW}. Let $M$ and $N$ be objects of $\cP\cB_k$, with gauges $\cG(M)$ and $\cG(N)$, respectively. 
Then  $M \wt^c_k N$ is an object of $\cP\cB_k$ with gauge 
$$\{P \wt^c_k Q = P \wt^u_k Q\}_{P,Q} 
$$
where $P \in \cG(M)$, $Q \in \cG(N)$. In particular, 
\beq \label{timesc}
M \wt^c_k N =  \limind_{P,Q} P \wt^u_k Q \;.
\eeq
\end{prop}
\begin{proof} The case of $k$ discrete follows from Remark~\ref{psudodiscr}. We then assume that $k$ is not discrete. 
The fact that $P \wt^c_k Q = P \wt^u_k Q$ is $\mathit 1$ of Lemma~\ref{basechange11}. We prove the first part of the statement.  We first need to show that 
\begin{lemma}  \label{tenspsudcan1}
Any  morphism $P \wt_k^u Q \to P' \wt_k^u Q'$ in \eqref{timesc}, for $P \subset P'$ in $\cG(M)$ and $Q \subset Q'$ in $\cG(N)$, is an open embedding. So,  $P \wt_k^u Q$ is an open sub-object  
  of  $P' \wt_k^u Q'$ 
carrying the canonical topology. 
\end{lemma}
\begin{proof}
  From  $\mathit 1$ of Lemma~\ref{basechange11}, we know that both source and target carry the  canonical topology.
To show that   $P \wt^u_k Q \to P' \wt^u_k Q'$ is an open sub-object  
we first need  to show that, for any open ideal $I$ of $k$, the map 
  $$P \wt^u_k Q/\ol{I(P \wt^u_k Q)} \to P' \wt^u_k Q'/\ol{I(P' \wt^u_k Q')}  $$
is injective.   But this map coincides with 
\beq
\label{flatness}
P/\ol{IP} \otimes_{k/I} Q/\ol{IQ} \to P'/\ol{IP'} \otimes_{k/I} Q'/\ol{IQ'} \;.
\eeq
The latter  is injective for the following reasons. First of all both maps 
 $$P/\ol{IP}  \to P'/\ol{IP'}  \;\;\mbox{and}\;\; Q/\ol{IQ} \to  Q'/\ol{IQ'}
  $$
are injective, because the subspace topology of $P$ in $P'$ (resp. of $Q$ in $Q'$) is the canonical topology of $P$ (resp. $Q$), which means that $\ol{IP'} \cap P = \ol{IP}$ (resp. $\ol{IQ'} \cap Q = \ol{IQ}$).
Then, by the flatness assumptions, the map \eqref{flatness} is injective.  
We then need to check that 
for any $J \subset I$ in $\cP(k)$, and for $I$ sufficiently small, 
$$
P/\ol{JP} \otimes_{k/J} Q/\ol{JQ}  \to  P/\ol{IP} \otimes_{k/I} Q/\ol{IQ}$$
and 
$$ P'/\ol{J'P} \otimes_{k/J} Q'/\ol{JQ'}  \to  P'/\ol{IP'} \otimes_{k/I} Q'/\ol{IQ'} 
$$
are both surjective. But this is clear by right-exactness of the tensor product in $\cM od_k$.  
 This implies at the same time 
 that  the  open sub-object $P \wt^u_k Q$ of $P' \wt^u_k Q'$ carries the canonical topology. 
 \end{proof}
We now get back to the proof of Proposition~\ref{tenspsudcan}.
Condition $\mathit 1$ of Definition~\ref{pseudocan} holds by Lemma~\ref{tenspsudcan1}. 
 Conditions $\mathit 2$ holds because, for any open ideal $I$ of $k$,  
 $$P \wt^u_k Q/\ol{IP \wt^u_k Q + P \wt^u_k IQ} = P/\ol{IP} \otimes_{k/I} Q/\ol{IQ}$$
  is a flat $k/I$-module. Part
$\mathit 3$  of that definition follows from \eqref{timesc}.
But  \eqref{timesc} follows from \eqref{billim0} and the fact that $P \wt^u_k Q = P \wt^c_k Q$ proven in $\mathit 1$ of  
Lemma~\ref{basechange11}. 
\end{proof}
 
\begin{prop} \label{basechange1}  
Let $A$   
be an object of $\cR\cR_k$ and $M$ be an object of $\cP\cB_k$. \hfill
\ben
\item  If $A$ is pro-flat over $k$,  $(M)^c_{A}$ 
 is an object of $\cP\cB_A$.  
 \item
 If $k$ satisfies {\bf OPW}, then $(\cP\cB_k , \wt^c_k )$ is a monoidal category.
\item If both $k$ and $A$  satisfy {\bf OPW} and $A$ is pro-flat over $k$,
the   functor $M \mapsto (M)^c_{A}$ induces an additive functor of monoidal categories 
$$(\cP\cB_k , \wt^c_k ) \to (\cP\cB_A , \wt^c_A) \;. $$
\een
 \end{prop}
\begin{proof} Let $M$ be an object of $\cP\cB_k$ and  $\cG$   be a $k$-gauge  of $M$.   
We want to show that the family  
$$
(\cG)^c_{A} := \{(P)^c_{A} = (P)^u_{A}  \,|\, P \in \cG\,\} 
$$
is an $A$-gauge in $(M)^c_{A}$. 
For any $P \in \cG$, let $i_P : P \to M$ be the open embedding of 
$P$ in $M$.  
 The fact that $(i_P)^c_{A} : (P)^c_{A} \to (M)^c_{A}$ is an open sub-object  follows from the fact 
that if $i_{P,P'} : P \to P'$ is  the open embedding of  
$P$ in $P'$, for $P' \supset P$ in $\cG$,  then  $(i_P)^c_{A} : (P)^c_{A} \to (P')^c_{A}$ is an open embedding as shown in part 
$\mathit 4$ of Lemma~\ref{basechange11}. In fact, as indicated in Lemma~\ref{basechange} and  in part $\mathit 1$ of  Proposition~\ref{indrep},   the functor $(-)^c_{A}$ of \eqref{scalext}  
commutes to inductive limits in $\cL\cM_k$ and  $\cL\cM_A$. So, $(M)^c_{A} = \limIND{P \in \cG}  (P)^c_{A}$ is a filtering inductive limit of open embeddings, hence it is the increasing union of the open sub-objects $ (P)^c_{A}$, for $P \in \cG$. 
Conditions $\mathit 1$ and $\mathit 2$ of  Definition~\ref{pseudocan} follow from  parts $\mathit 2$ and $\mathit 3$ of Lemma~\ref{basechange11}, respectively. 
Condition $\mathit 3$ has already been proven. 
\par Part $\mathit 2$ follows from part $\mathit 1$ of Lemma~\ref{basechange11}, taking into account part $\mathit 1$ of Proposition~\ref{indrep}.
\par Part $\mathit 3$ is clear.
\end{proof}
\end{section}
\begin{section}{Representable functors}

\begin{defn}  \label{indprec0}
We will denote by  $\cI\cR\cR_k$ the full subcategory of  $\cL\cR_k$ whose objects $R = \limind^\cL_\alpha R_\alpha = \limind^u_\alpha R_\alpha$  
are inductive limits in both  $\cL\cR^u_k$ and  $\cL\cR_k$ of the same  inductive system
$\{R_\alpha\}_{\alpha \in A}$ in $\cR\cR_k$.
 \end{defn} 
 \begin{rmk} \label{indprec1}
 $\cI\cR\cR_k$ is in fact a full subcategory of  $\cL\cR^u_k$, and its objects are those  inductive limits in  $\cL\cR_k$  of   inductive systems
$\{R_\alpha\}_{\alpha \in A}$ in $\cR\cR_k$, which happen to be objects of $\cL\cR^u_k$.
\end{rmk}
 
 \begin{lemma} \label{basechIRR}
 Let $\{R_\alpha\}_{\alpha \in A}$ be an inductive system  in $\cR\cR_k$. Assume 
 $$ {\limind}^\cL_\alpha R_\alpha = {\limind}^u_\alpha R_\alpha \;.$$  
 Then for any object  $C$  of $\cR\cR_k$ 
  $$ {\limind}^\cL_\alpha (R_\alpha)^u_C = {\limind}^u_\alpha (R_\alpha)^u_C \;.$$  
 \end{lemma}
 \begin{proof}
 We need to show that ${\limind}^\cL_\alpha (R_\alpha)^u_C$ is $C$-uniform. 
By assumption, we have an isomorphism 
 $$
 \limPRO{\sJ \in \sC(\{R_\alpha\}_{\alpha \in A})} {\limIND\alpha} R_\alpha/J_\alpha 
\iso  \limPRO{I \in \cP(k)} {\limIND\alpha} R_\alpha/\ol{I R_\alpha} 
 $$
 in $\cL\cR^u_k$.  So, for any coherent system of open ideals $\sJ = (J_\alpha)_\alpha$ 
 there exists an open ideal $I \in \cP(k)$ such that $J_\alpha  \supset I R_\alpha$, for any $\alpha \in A$. 
\par
Let now $\sH = (H_\alpha)_\alpha$ be a coherent system of open ideals in the inductive system 
$(R_\alpha \wt^u_k C)_{\alpha \in A}$ in $\cR\cR_C$. 
 We need to prove that the canonical map provides an isomorphism
\beq \label{isolim} \begin{split}
& \limPRO{\sH \in \sC((R_\alpha \wt^u_k C)_{\alpha \in A})} {\limIND\alpha} (R_\alpha \wt^u_k C)/H_\alpha 
\iso \\ \limPRO{J \in \cP(C)} {\limIND\alpha} & (R_\alpha \wt^u_k C/J )=  \limPRO{J \in \cP(C)} {\limIND\alpha} ((R_\alpha/\ol{(J \cap k) R_\alpha} ) \otimes_{k/(J \cap k)} C/J )  \;.
\end{split}
\eeq 
Again \eqref{isolim} will hold if for any $\sH = (H_\alpha)_\alpha$ as before, there exists an open ideal $J \in \cP(C)$ such that 
$H_\alpha \supset R_\alpha \wt^u_k J$, for any $\alpha \in A$. Equivalently, we may replace $H_\alpha$ by its inverse image in $R_\alpha \otimes_k C \to R_\alpha \wt^u_k C$. We 
keep the name $H_\alpha$  for that inverse image. 
Now, for any $\alpha \in A$,  
$$
R_\alpha \wt^u_k C = \limPRO{J \in \cP(R_\alpha), Q\in \cP(C)} R_\alpha/J  \otimes_k  C/Q
$$
so that we may assume that  $H_\alpha$ has the form
$$
H_\alpha = J_\alpha \otimes C + R_\alpha \otimes_k Q_\alpha
$$
for some $J_\alpha \in \cP(R_\alpha)$ and $Q_\alpha \in \cP(C)$. The  coherence condition becomes 
$$
(j_{\alpha,\beta} \otimes_k C)^{-1} (J_\beta \otimes C + R_\beta \otimes_k Q_\beta ) =  J_\alpha \otimes C + R_\alpha \otimes_k Q_\alpha 
$$
for any $\alpha \leq \beta$. But $j_{\alpha,\beta} \otimes_k C$ is the continuous $C$-linear extension of $j_{\alpha,\beta} : R_\alpha \to R_\beta$, so that 
$$
(j_{\alpha,\beta} \otimes_k C)^{-1} (J_\beta \otimes C + R_\beta \otimes_k Q_\beta ) =  j_{\alpha,\beta}^{-1}(J_\alpha) \wt^u_k C +  R_\alpha  \wt^u_k Q_\beta \;.
$$
This shows that there is an ideal $Q \in \cP(C)$ such that $H_\alpha \supset R_\alpha  \otimes_k Q $, for any $\alpha$. On the other hand, 
$\sJ := (J_\alpha)_\alpha$ is a coherent system of open ideals in $(R_\alpha)_\alpha$, so that, as we saw before,  there exists an open ideal $I \in \cP(k)$  such that $I R_\alpha \subset J_\alpha$ for any $\alpha$. 
We conclude that for any coherent system of open ideals $\sH = (H_\alpha)_\alpha$ as before, 
$$H_\alpha \supset  I R_\alpha \otimes_k C +  R_\alpha \otimes_k  Q \;.
$$
  \end{proof}
An object $\sL_k = {\limIND{\alpha \in A}}^\cL R_\alpha$ of $\cI\cR\cR_k$ determines the functor 
 $F : \cL\cR_k \to \cM od_k$ given by  
\beq
\label{reprdef}
X \mapsto  \hom_{\cL\cR_k} ({\limIND{\alpha \in A}} R_\alpha, X) = \limPRO{\alpha \in A} \hom_{\cL\cR_k}(R_\alpha, X) \;.
\eeq  
 We say that  $\sL_k$ \emph{represents} the functor $F$. 
Recall that  $\{(R_\alpha)_{\alpha \in A}, (\iota_{\alpha,\beta}:R_\alpha \to R_\beta)_{\alpha \leq \beta}\}$ is an inductive system in $\cR\cR_k$ and that 
$\cL_k$ is an object of $\cL\cR_k^u$. By Yoneda's lemma,  $\sL_k$ is determined by the restriction of the functor $F$ to $\cL\cR_k^u$.
\begin{prop} \label{scalrestr} Let    $F : \cL\cR_k \to \cM od_k$ be defined by \eqref{reprdef} as above. Then, 
for any object $C$ of $\cR\cR_k$ the restriction of $F$ to  $\cL\cR^u_C$  is  represented by the object 
$$ \sL_C :=
{\limIND{\alpha \in A}}^\cL (C \wt^u_k R_\alpha) = {\limIND{\alpha \in A}}^u (C \wt^u_k R_\alpha) 
$$
of $\cI\cR\cR_C$. 
\end{prop}
\begin{proof}
In fact, any object $X$ of $\cL\cR^u_C$ may be viewed as an object of $\cL\cR^u_k$. 
We have
\beqa \begin{split}
F(X) =&   \hom_{\cL\cR_k} ({\limIND{\alpha \in A}}^\cL R_\alpha, X) = \limPRO{\alpha \in A} \hom_{\cL\cR_k}(R_\alpha, X) =  \\
&\limPRO{\alpha \in A} \hom_{\cL\cR_C}(C \wt^u_k R_\alpha, X) = \hom_{\cL\cR_C} ({\limIND{\alpha \in A}}^\cL C \wt^u_k R_\alpha, X) \;.
\end{split}
\eeqa
\end{proof}
\begin{rmk} \label{scalrestr1}
This shows that, for any $C$ in $\cR\cR_k$,  the functor 
\beq \label{indtens}
\begin{split}
(-)^\ind_C : \cI\cR\cR_k &\longrightarrow \cI\cR\cR_C \\
{\limIND{\alpha \in A}}^\cL R_\alpha & \longmapsto {\limIND{\alpha \in A}}^\cL (R_\alpha)^u_{/C}
\end{split}
\eeq
is well defined. 
\end{rmk}

\begin{cor} \label{banach2} We assume here that $K$ is a non-archimedean field and that $k = K^\circ$. We have
\ben
\item Let $M$ and $N$ be pseudobanach $K$-spaces with gauges $\cG$ and $\cH$, respectively, then $M \otimes^c_k N$ is a pseudobanach $K$-space with gauge $\cG \otimes \cH$ consisting of the 
family of the open sub-objects $P \wt^u_k Q$, for $P \in \cG$ and $Q \in \cH$.   
 \item If $K$ is non-trivially valued the category $\cP\cB_k$ is equivalent to the category $\cB an_K$ of $K$-Banach spaces and continuous $K$-linear maps.  The
 tensor product $M \wt^c_k N$ corresponds to both  the (separated complete) projective and  injective tensor products $ M \otimes_{K,\pi}  N =  M \otimes_{K,\iota}  N$ of \cite[\S 17 B]{schneider}.
 \item If $K$ is trivially valued so that $K^\circ = K$, $\cP\cB_K$ is the full subcategory of $\cL\cM_K$ consisting of $K$-vector spaces equipped with the discrete topology and $-\wt_K^c - =   - \otimes_K -$.
\een
\end{cor}
\end{section} 
 \begin{section}{Rings of $\fpm$-type} \label{pmtype}
We recall that a subset 
$T$ of a topological ring $R$ is \emph{bounded} if, for any neighborhood $U$ of $0$ in $R$, there exists a neighborhood $V$ of $0$ in $R$ such that
$V\,T \subset U$.  An element $x \in R$ is \emph{power bounded} if the set $T_x = \{1,x,x^2,\dots\}$ is bounded. 
If $R$ is a $k$-ring and the topology of $R$ is $k$-linear, then $T$ is bounded if and only if 
the $k$-sub-module of $R$ generated by $T$ is bounded. Under the same assumptions, an element $x \in R$ is power bounded 
if and only if the $k$-sub-ring $k[x]$ of $R$ is bounded.
For an object $R$ of  $\cL\cR_{k}$ we denote 
by $R^\circ$ the subset of $R$ consisting of power bounded elements. 
\begin{lemma}
For any object $R$ of  $\cL\cR_{k}$, $R^\circ$ is a subring of $R$. 
\end{lemma}
\begin{proof}
Let $\cP(R)$ be a fundamental system of open $k$-submodules of $R$, and let $x ,y \in R^\circ$. 
We show that both $T_{x+y}$ and $T_{xy}$ are bounded. 
For any $U \in \cP(R)$, let $V  \in \cP(R)$ be such that   $V  k[y] \subset U$. Let then $W   \in \cP(R)$  be such that $W k[x] \subset V$. Then 
$W k[x]  k[y] \subset U$. So, $k[x]  k[y]$ is bounded. The $k$-sub-module of $R$ generated by $k[x]  k[y]$ contains both 
$T_{x+y}$ and $T_{xy}$.  
\end{proof}
\begin{defn} 
We say that an object $R$ of  $\cL\cR_{k}$ is \emph{of $\fpm$-type}  if $R^\circ$ is an open  subring of $R$ and its subspace topology 
is $R^\circ$-linear. 
\end{defn}
\begin{rmk}
Notice that,  if
  $R$ is of $\fpm$-type, then
 $R^\circ$ is an  object of $\cR\cR_{k}$ and it is the (unique) maximal open subring $A$ of $R$ such that the subspace topology of $A$  is $A$-linear.  
 \end{rmk}
 \begin{lemma} Let   $R$ be an object of  $\cL\cR_{k}$. 
 For any open 
 subring $A$ of $R$ the condition that the subspace topology of $A$  be $A$-linear
 is equivalent to $A$ being bounded. 
 \end{lemma}
 \begin{proof}
 In fact, if the subspace topology on  $A$ is $A$-linear and $U$ is any neighborhood  of $0$ in $R$, then 
 there exists an open ideal $J$ of $A$ such that $J \subset U \cap A$, hence $J\,A \subset U$. Conversely, if $A$ is bounded and $U$ is any neighborhood  of $0$ in $A$, 
 there exists a neighborhood $V$ of $0$ in $A$ such that
$V\,A \subset U$. But then $V\,A$ is an ideal of $A$ and a fundamental system of open $k$-submodules of $A$ may be assumed to consist of open ideals of $A$. 
\end{proof}
We conclude 
\begin{cor}
An object $R$ of $\cL\cR_k$ is of $\fpm$-type if and only if $R^\circ$ is an open bounded $k$-subring of $R$.
\end{cor}
\end{section}
\begin{section}{Multivalued rings}  \label{multival}
We now specialize the definitions of the previous sections to the classical case of topologies defined by a family of semivaluations. 
\par
We assume in this section  that $K = (K,v)$ is a non-archimedean field and that $k = K^\circ$.
\par \smallskip
\begin{defn} \label{valdef}
A \emph{semivaluation} on a ring $R$ is a map $w:R \to \R  \cup \{+\infty\}$ such that 
$$w(0) = +\infty \;\;,\;\; w(x+y) \geq \min(v(x),v(y)) \;\;,\;\; w(xy) \geq w(x) + w(y) \;,
$$
 for any $x,y \in R$; $w$ is \emph{separated} (resp.  \emph{power-multiplicative}, resp. \emph{multiplicative}) if $w(x) \neq +\infty$ for $x \neq 0$ (resp. $w(x^n) = n w(x)$, 
 for any $x \in R$ and $n \in \Z$, resp.  $w(xy) = w(x) + w(y)$, for any $x,y \in R$).  A  multiplicative semivaluation is called 
a \emph{pseudovaluation}, and  a \emph{valuation} if moreover it is separated.  
A
\emph{$k$-Banach  ring} is an object of $\cL\cR_k$ 
whose topology is induced by a single semivaluation $w$ extending the valuation $v$. Then $w$ is necessarily separated. A morphism of 
$k$-Banach  rings is a continuous $k$-algebra morphism, \ie a morphism in the category $\cL\cR_k$. 
\end{defn} 
 \par
We denote by $w^{\sp}$ the \emph{spectral valuation} associated to $w$, that is 
\beq \label{specdef}
w^\sp(f) = \lim_n \frac{1}{n} w(f^n) = \sup_n \frac{1}{n} w(f^n) \;.
\eeq
Then $w^\sp$ is in fact a power multiplicative semivaluation on $R$ and
$$
R^\circ = \{x \in R \,|\, w^\sp(x) \geq 0\,\} \;.
$$
So, a  $k$-Banach ring $R$ for the semivaluation $w$ is of $\fpm$-type if and only if 
$w$ and $w^\sp$ induce the same topology on $R$, or, equivalently, if its topology can be defined by 
a power-multiplicative semivaluation $w=w^\sp$.  
If $(R_1,w_1), \dots , (R_N,w_N)$ are   $k$-Banach rings, so is 
$$R := (R_1,w_1) \wt^c_k \dots  \wt^c_k  (R_N,w_N) \;.
$$
The proof is essentially identical to the one of \cite[Lemma 17.2]{schneider}.
More precisely,  $R$ is the separated completion of $R_1 \otimes_k \dots \otimes_k R_N$ in the \emph{product semivaluation} $w$ defined, for any 
$y \in   R_1 \otimes_k \dots \otimes_k R_N$
  by 
\beq \label{prdoval}
w(y) = \sup \min_{j=1,\dots,M} w_1(y_{j,1}) +\dots +w_N(y_{j,N})
\eeq
where the supremum is taken over all representations 
$$
y = \sum_{j=1}^M y_{j,1} \otimes \dots \otimes y_{j,N}
$$
with $y_{j,i} \in R_i$, for $i=1,\dots,N$. 
\par

A \emph{multivalued $k$-ring} is an  
  object $R$ of $\cL\cR_k$ whose topology is induced by a family of  semivaluations $\{w_r\}_{r \in T}$. 
   We write $R = (R, \{w_r\}_{r \in T})$.
 We are especially interested in the case when $k =K^\circ$, for a non-archimedean field $(K,v_K)$ (possibly trivially valued) and the  semivaluations  $w_r$, for $r \in T$, induce $v= v_K$ on $k$.

 \begin{rmk}\label{gauss} let $(K,v_K)$ be a non-trivially valued non-archimedean field. Then a commutative $K$-Banach algebra in the classical sense is a $K^\circ$-Banach ring  $(A,w)$ for which the scalar product 
 $K^\circ \times A \to A$  extends to a structure of  $K$-vector space on $A$. For any commutative $K$-Banach algebra $A$, one defines classically a subring  
 $A\{T\}$ of $A[[T]]$ as the subset of power series $\sum_{i=0}^{+\infty} a_i T^i$ such that 
 $\lim_{i \to +\infty} a_i = 0$. Then  $A\{T\}$ equipped with the semivaluation $w_T$ such that
  $$
 w_T(\sum_{i=0}^{+\infty} a_i T^i ) = \inf_{i \geq 0} w(a_i) \;,
 $$
 is a $K$-Banach algebra.  The semivaluation $w_T$ is called the \emph{Gauss valuation}. 
 This operation, and terminology, can be iterated in more variables.  It is the construction used in \cite[2.1.7]{Berkovich}.
 \end{rmk}
\end{section}

\begin{section}{Pseudobanach algebras} \label{upcalgebras} 
 \begin{defn}   \label{banringdef}
A \emph{pseudobanach} $k$-ring is any object $A$ of $\cL\cR_k \cap \cP\cB_k$ which admits a gauge $\cG$ such that one element $R$ of $\cG$ is a $k$-subring of $A$. 
A pseudobanach $k$-ring is \emph{of $\fpm$-type} if it is of $\fpm$-type as an object of $\cL\cR_k$.
We denote by $\cP\cB\cA_k$ (resp. $\cU\cP\cB\cA_k$)  the full subcategory of $\cL\cR_k$ whose objects are  pseudobanach  $k$-rings (resp. of $\fpm$-type). 
\end{defn} 
\begin{rmk}   \label{banringdef0}  Let $A$  be a  pseudobanach $k$-ring. 
Let $\cG(A)$ be a  
gauge  of $A$ such that $R \in \cG(A)$ is a subring of $A$.  
The  subring $A^\circ$  of $A$ is open since it contains $R$.
So, \emph{a pseudobanach $k$-ring $A$ is of $\fpm$-type if and only if $A^\circ$ is bounded}. In that case, we may assume that $A^\circ \in \cG$. Conversely, if 
$A^\circ \in \cG$, $A$ is of $\fpm$-type. 
\end{rmk} 

\begin{rmk}   \label{banringdef1} 
For $A,B$ in 
 $\cP\cB\cA_k$ the object $A \wt^c_k B$ of  $\cP\cB_k$  is canonically equipped with a structure of a   pseudobanach   $k$-ring    via the structure of $\mathit 1$ of 
 Lemma~\ref{tensrep}
\end{rmk}
Let $A$ be an object of $\cP\cB\cA_k$ with gauge $\cG(A)$ where  $R \in \cG(A)$ is a subring of $A$. We observe that a power series 
 $$\sum_{\ul{u} \in \Z_{\geq 0}^n} a_{\ul{u}} \ul{T}^{\ul{u}} \in A[[\ul{T}]] = A[[T_1,\dots,T_n]] 
 $$
 is restricted \cite[Chap. III, \S 4, n. 2, Def. 2 p. 253]{algebracomm} \ie is such that for any $P \subset R$, $P \in \cG(A)$,  $a_{\ul{u}} \in P$ for almost all $\ul{u} \in \Z_{\geq 0}^n$,  if and only if  $a_{\ul{u}}  \to 0$ as $|\ul{u}| := \sum_{i=1}^n u_i \to +\infty$. It follows from the definition of a $k$-gauge  that the product of two restricted power series in $A[[\ul{T}]]$ is restricted. 
We denote by $A\{\ul{T}\}$ the subring of $A[[\ul{T}]]$ consisting of restricted power series. 
 The  family of $k$-submodules 
of $A\{\ul{T}\}$
\beq
U\{\ul{T}\} = \{ \sum_{\ul{u} \in \Z_{\geq 0}^n} a_{\ul{u}} \ul{T}^{\ul{u}} \in A\{\ul{T}\}\,|\, a_{\ul{u}} \in U \,,\, \forall \ul{u} \in  \Z_{\geq 0}^n\,\}
\eeq
for $U \in \cG(A)$, is a $k$-gauge in $A\{\ul{T}\}$. Obviously $R\{\ul{T}\} \in \cG(A\{\ul{T}\})$ is a subring of $A\{\ul{T}\}$. It then follows that $A\{\ul{T}\}$ is an object of 
 $\cP\cB\cA_k$
 \begin{defn}\label{pseudoaffinoid} Let $A$ be an object of $\cP\cB\cA_k$ with gauge $\cG(A)$ where  $R \in \cG(A)$ is a ring. Then we regard the
 ring of restricted power series   with coefficients in $A$
 as an object  $A\{\ul{T}\}$ of $\cP\cB\cA_k$, with gauge $\cG(A\{\ul{T}\})$ and $R\{\ul{T}\} \in \cG(A\{\ul{T}\})$ an open subring of $A\{\ul{T}\}$. 
 \end{defn}
 \begin{lemma}\label{pseudoaffinoid1}  Let $A$ be an object of $\cP\cB\cA_k$. Then 
 $$A\{\ul{T}\} = A \wt^c_k k\{\ul{T}\}\;.$$
If $K$ is a non-archimedean field, $k = K^\circ$, and $A$ is a commutative $K$-Banach algebra identified with an object of $\cP\cB\cA_k$, $A\{\ul{T}\}$ is a commutative $K$-Banach algebra and 
$$
A\{\ul{T}\} = A \wt_{K,\pi} K\{T\} = A \wt_{K,\iota} K\{T\} 
$$
where $- \wt_{K,\pi} -$ (resp. $- \wt_{K,\iota} -$) is the (separated complete) projective (resp. injective) tensor product of \cite{schneider}. If the topology of $A$ (resp. $K\{\ul{T}\}$) 
is induced by a semivaluation $w$ (resp. by the Gauss semivaluation) the topology of $A\{\ul{T}\}$ is induced by the product semivaluation \eqref{prdoval} which coincides with the Gauss valuation of $A\{\ul{T}\}$ of Remark~\ref{gauss}.
 \end{lemma}
 \begin{exa}\label{banach} We observe that Definition~\ref{pseudoaffinoid1} generalizes both the classical definition of ring of restricted power series with coefficients
in a linearly topologized ring of \cite[Chap. III, \S 4, n. 2, p. 252-259]{algebracomm} and the definition of an  $A$-affinoid algebra  $A\{T\}$ in the sense of \cite[2.1.7]{Berkovich}, when $A$ itself is a commutative $K$-Banach algebra. 
 \end{exa}
\begin{lemma} \label{produnif}  For $A$, $B$ in   $\cU\cP\cB\cA_k$,  $A \wt^c_k B$ is in $\cU\cP\cB\cA_k$.
\end{lemma}
\begin{proof} We already pointed out in Remark~\ref{banringdef1} that $A \wt^c_k B$ is in $\cP\cB\cA_k$ and then, by Remark~\ref{banringdef0}, 
$(A \wt^c_k B)^\circ$ is open. (On the other hand  the latter fact  also follows from $A^\circ \wt_k^u B^\circ \subset (A \wt^c_k B)^\circ$.)
Let $\cG(A)$ (resp. $\cG(B)$) be a gauge for $A$ (resp. $B$) containing an element $R$ (resp. $S$) which is a subring of $A$ (resp. $B$). 
 Assume that $ (A \wt^c_k B)^\circ$ is unbounded. Then, for any  open ideal $I$ of 
$k$, 
 there exists $x_I \otimes y_I\in (A \wt^c_k B)^\circ$
such that $I x_I \not\subset R$ or $I y_I \not\subset S$. But $x_I \in A^\circ$ and $y_I \in B^\circ$, so this would violate the boundedness of either $A^\circ$ or 
$B^\circ$, absurd. 
\end{proof}

\begin{defn} \label{classcomp} Let $K$ be a non-archimedean field and let $k = K^\circ$.   
A \emph{pseudobanach}   $K$-algebra is any object 
of $\cP\cB\cA_k$   which is a $K$-vector space.  
We denote by $\cB an \cA lg_K$ (resp. $\cU\cB an \cA lg_K$) the full subcategory of $\cL\cR_k$ whose objects are pseudobanach  $K$-algebras (resp. of $\fpm$-type). 
\end{defn}
\begin{exa} 
If $K$ is non-trivially valued, then the category $\cB an\cA lg_K$  is equivalent to the 
commonly used category  of $K$-Banach algebras, with the only \emph{caveat} that   morphisms   are simply continuous $K$-linear morphisms.  
The category $\cU \cB an \cA lg_K$ is equivalent to the category of $K$-Banach algebras of $\rm pm$-type and continuous $K$-algebra morphisms, considered by Fontaine \cite{fontaineA}.
\par 
\end{exa}

 It follows from Proposition~\ref{basechange1} that
\begin{cor} \label{basechunif}   Let $\kappa$   
be an object of $\cR\cR_k$.  
The base-change functor $A \mapsto (A)^c_{/\kappa}$,  induces   additive functors  of monoidal categories 
$$(\cP\cB\cA_k, \wt^c_k) \to (\cP\cB\cA_{\kappa} , \wt^c_{\kappa} ) 
\;\;\mbox{and}\;\;\;(\cU\cP\cB\cA_k, \wt^c_k) \to (\cU\cP\cB\cA_{\kappa} ,  \wt^c_{\kappa} ) 
\;.$$
\par
Assume in particular $L/K$ is an extension of non-archimedean valued fields. Then, for $k = K^\circ$ and $\kappa = L^\circ$, the base-change functor $A \mapsto A_{/\kappa}$ induces $K$-linear functors 
of monoidal categories
 $$(\cB an \cA lg_K, \wt^c_{K^\circ}) \to (\cB an \cA lg_L , \wt^c_{L^\circ}) \;\; \mbox{and}\;\; (\cU \cB an\cA lg_K, \wt^c_{K^\circ}) \to (\cU \cB an\cA lg_L  , \wt^c_{L^\circ})
 \;.
 $$
\end{cor} 
\end{section}

\end{document}